\documentclass[a4paper,11pt,twoside]{article}
\usepackage{color}
\usepackage{verbatim}
\usepackage{theorem,amsmath,amssymb,amscd}
\usepackage{mathdots}
\usepackage{booktabs}
\usepackage{xcolor}
\usepackage[hyperfootnotes=false, colorlinks, linkcolor={blue}, citecolor={magenta}, filecolor={blue}, urlcolor={blue}]{hyperref}

\usepackage{fancyhdr}
\theoremstyle{change}
\allowdisplaybreaks
\newcommand{\Hl}{{\mathbb H}}
\newcommand{\C}{{\mathbb C}}
\newcommand{\R}{{\mathbb R}}
\newcommand{\Q}{{\mathbb Q}}
\newcommand{\Z}{{\mathbb Z}}

\newcommand{\g}{{\mathfrak g}}
\newcommand{\p}{{\mathfrak p}}

\newcommand{\n}{{\mathfrak n}}

\newcommand{\la}{{\mathfrak a}}
\newcommand{\mk}{{\mathfrak k}}

\newcommand{\A}{{\mathbb A}}

\newcommand{\sym}{\operatorname{sym}}

\newcommand{\cD}{{\cal D}}
\newcommand{\cF}{{\cal F}}
\newcommand{\cG}{{\cal G}}
\newcommand{\cH}{{\cal H}}
\newcommand{\cP}{{\cal P}}

\newcommand{\cL}{{\cal L}}
\newcommand{\cM}{{\cal M}}
\newcommand{\cN}{{\cal N}}

\newcommand{\cW}{{\cal W}}

\newcommand{\diag}{\operatorname{diag}}

\newcommand{\SL}{{\rm SL}}

\newcommand{\Hom}{{\rm Hom}}

\newcommand{\qed}{\hspace*{\fill}\rule{1ex}{1ex}}

\numberwithin{equation}{section}
\theoremstyle{plain}
 \newtheorem{thm}{Theorem}[section]
 \newtheorem{prop}[thm]{Proposition}
 \newtheorem{lem}[thm]{Lemma}
 \newtheorem{cor}[thm]{Corollary}

 \newtheorem{defn}[thm]{Definition}
 \newtheorem{rem}[thm]{Remark}

 \newenvironment{proof}{\vspace{1ex}\noindent{\it Proof.}\hspace{0.1em}}
	{\hfill\qed\vspace{2ex}}

\begin{document}
\title{An explicit construction of non-tempered cusp forms on $O(1,8n+1)$}
\author{Yingkun Li, Hiro-aki Narita and Ameya Pitale}
\date{}
\maketitle
\begin{abstract}
 We explicitly construct cusp forms on the orthogonal group of signature $(1,8n+1)$ for an arbitrary natural number $n$ as liftings from Maass cusp forms of level one. In our previous works \cite{P} and \cite{Mu-N-P} the fundamental tool to show the automorphy of the lifting was the converse theorem by Maass. In this paper, we use the Fourier expansion of the theta lifts by Borcherds \cite{B} instead. 

 We also study cuspidal representations generated by such cusp forms and show that they are irreducible and that all of their non-archimedean local components are non-tempered while the archimedean component is tempered, if the Maass cusp forms are Hecke eigenforms. The standard $L$-functions of the cusp forms are proved to be products of symmetric square $L$-functions of the Hecke-eigen Maass cusp forms with shifted Riemann zeta functions.
\end{abstract}
\section{Introduction}
A unique feature of automorphic forms or representations of reductive groups of higher rank~(or of larger matrix size) is the existence of non-tempered cusp forms or cuspidal representations, namely cuspidal representations which have a non-tempered local component. 
Due to such existence the Ramanujan conjecture for $GL(2)$ can not be generalized to a general reductive group in a straightforward manner. In fact, according to the generalized Ramanujan conjecture for quasi-split reductive groups, such generalization would be possible if we impose the ``genericity'' on cuspidal representations, namely they are assumed to admit Whittaker models. 
It seems that the existence of non-tempered cuspidal representations has often been an obstruction to establish a general result of automorphic representations. Hence it is of fundamental importance to study non-tempered cusp forms or cuspidal representations in detail. 

A well-known expected method for the construction of non-tempered cusp forms or representations is a lifting from a smaller group, e.g.\ a lifting from $GL(2)$. As related works in the case of holomorphic automorphic forms, we cite Kurokawa \cite{Ku}, Oda \cite{Od}, Rallis-Schiffman \cite{Ra-Sch}, Sugano \cite{Su}, Ikeda \cite{Ik1}, \cite{Ik2}, Yamana \cite{Ya} and Kim-Yamauchi \cite{Ki-Yam} et al. We are interested in such liftings {for} the case of non-holomorphic real analytic automorphic forms. We already have \cite{P} and \cite{Mu-N-P}  for the groups of low rank or of small matrix size. 
In this paper, for a general $n$, we provide an explicit lifting construction of non-tempered cusp forms or cuspidal representations on the orthogonal group $O(1,8n+1)$ over {$\Q$} with non-holomorphic real analytic automorphic forms viewed as Maass cusp forms on real hyperbolic spaces of dimension $8n+1$. 

Let $(\Z^{8n},S)$ be an even unimodular lattice defined by a positive definite matrix $S$ of degree $8n$. We realize the orthogonal group $O(1,8n+1)$ by the non-degenerate symmetric matrix $Q=
\begin{pmatrix}
 & & 1\\
& -S & \\
1 & &
\end{pmatrix}$. Let $f$ be a Maass cusp form of level one with Fourier coefficients $\{ c(n): n \in \Z\}$. We introduce a function $F_f$ on the real Lie group $O(1,8n+1)(\R)$ or the real hyperbolic space of dimension $8n+1$ by a Fourier expansion whose Fourier coefficients are explicitly written in terms of those of $f$ as follows : for $\lambda \in \Z^{8n}$ set
\begin{equation}
  \label{Fouriercoeff-lifting}
  A(\lambda):=|\lambda|_S\sum_{d|d_{\lambda}}c\left(-\frac{|\lambda|_S^2}{d^2}\right)d^{4n - 2},
\end{equation}
where $d_{\lambda}$ denotes the greatest common divisor of the non-zero entries of $\lambda$ and $|\lambda|_S = \sqrt{q_S(\lambda)}$ with $q_S$ the quadratic form associated to $S$. Let $\Gamma_S$ be an arithmetic subgroup of $O(1,8n+1)(\R)$ defined by the maximal lattice $(\Z^{8n+2},Q)$.
\begin{thm}[Theorem \ref{Construction-lifting}]\label{Intro-Thm1}
The function $F_f$ is a cusp form on $O(1,8n+1)(\R)$ with respect to $\Gamma_S$. If $f$ is non-zero, so is $F_f$.
\end{thm}
In our previous works (\cite{P} and \cite{Mu-N-P}), we used the converse theorem due to Maass (\cite{Ma}) to obtain automorphy of the lift. A basic limitation of the Maass converse theorem is that it provides automorphy only with respect to a discrete subgroup generated by translations and one inversion. For the case of $n>1$, it seems difficult to determine the generators of $\Gamma_S$. Hence, the Maass converse theorem method, though applicable, does not give automorphy with respect to all of $\Gamma_S$.

To avoid this difficulty we apply a theta lift {from} the Maass form $f$, which yields an automorphic form $\Phi(\nu, f)$ {(for the notation see~\ref{Theta-lift})} on $O(1, 8n + 1)(\R)$ with respect to $\Gamma_S$. We get the automorphy of $F_f$ by explicitly computing the Fourier coefficients of $\Phi(\nu, f)$ using the calculation of Borcherds \cite[Theorem 7.1]{B}, and showing that they are exactly the same as $A(\lambda)$ defined above, i.e., $F_f$ is equal to the theta lift.

We next show that $F_f$ is a non-tempered cusp form, or $F_f$ generates a cuspidal representation of $O(1,8n+1)(\A)$ which has a non-tempered local component. For this, it is useful to obtain an adelic reformulation of the lift $F_f$ as a function on $O(1,8n+1)(\A)$. Our result is stated as follows:
\begin{thm}[Theorem~\ref{MainThm-second}]\label{Intro-Thm2}
Let $\pi_{F_f}$ be the cuspidal representation generated by $F_f$ and suppose that $f$ is a Hecke eigenform with the Hecke eigenvalue $\lambda_p$ for each finite prime $p$.\\
(1)~The representation $\pi_{F_f}$ is irreducible and thus has the decomposition into the restricted tensor product $\otimes'_{v\le\infty}\pi_v$ of irreducible admissible representations $\pi_v$ of $O(1,8n+1)(\Q_v)$.\\
(2)~For $v=p<\infty$, $\pi_p$ is the spherical constituent of the unramified principal series representation of $O(1,8n+1)(\Q_p)$ with the Satake parameter
\[
\diag \left( \left(\frac{\lambda_p+\sqrt{\lambda_p^2-4}}{2}\right)^2,p^{4n-1},\cdots,p,1,1,p^{-1},\cdots,p^{-(4n-1)},\left(\frac{\lambda_p+\sqrt{\lambda_p^2-4}}{2}\right)^{-2} \right).
\]
(3)~For every finite prime $p<\infty$, $\pi_p$ is non-tempered while $\pi_{\infty}$ is tempered.
\end{thm}

For the first assertion of this theorem we remark that, if $f$ is a Hecke eigenform, so is $F_f$~(cf.~Theorem \ref{MainThm-first}). By \cite[Theorem 3.1]{N-P-S} we then see that $F_f$ generates an irreducible cuspidal representation of $O(1,8n+1)(\A)$ after showing the irreducibility of the archimedean representation which $F_f$ generates. As a consequence of this theorem we have the result on the standard L-function of $\pi_{F_f}$ or $F_f$.
\begin{cor}[Corollary~\ref{Std-L-fct}]
For any prime $p$ the local $p$-factor $L_p(\pi_{F_f},\operatorname{St},s)$ of the standard $L$-function for $\pi_{F_f}$~(or $F_f$) is written as
\begin{align*}
L_p(\pi_{F_f},\operatorname{St},s)&=\zeta_p(s)(1-(\lambda_p^2-2)p^{-s}+p^{-2s})^{-1}\prod_{j=0}^{8n-2}\zeta_p(s+j-(4n-1))\\
&=L_p(\sym^2(f),s)\prod_{j=0}^{8n-2}\zeta_p(s+j-(4n-1)),
\end{align*}
where $\zeta_p$ denotes the $p$-factor of the Riemann zeta function and $L_p(\sym^2(f),s)$ is the $p$-factor of the symmetric square $L$-function for $f$.
\end{cor}
What is crucial to obtain the results above is Sugano's non-archimedean local theory of the ``Whittaker functions'' on orthogonal groups~(cf.~\cite[Section 7]{Su}), which is now known as   ``special Bessel models''~(cf.~\cite{No-Ps},~\cite{Fu-Mo}). We verify that Sugano's local theory is applicable to the Fourier coefficients of the adelic cusp form for $F_f$. 
We furthermore verify that they also satisfy the local Maass relation~(cf.~(\ref{local-Maass-rel}),~(\ref{local-Maass-rel-2})), which leads to a nice reduction of the calculation of the Hecke eigenvalues of $F_f$. Indeed, we see in a general setting that the space of the non-archimedean local Whittaker functions satisfying the local Maass relation admits a simple Hecke module structure~(cf.~Proposition \ref{Simple-Heckemodule}). As an application of this we obtain a simple explicit formula for the Hecke eigenvalues of $F_f$~(cf.~Theorem \ref{MainThm-first} (2)), from which we deduce Theorem \ref{Intro-Thm2} (2) (or Theorem \ref{MainThm-second}, {2}).

We discuss our results in terms of further problems. Sugano's local theory just mentioned is useful even if we replace an even umimodular lattice by a general maximal lattice. One remaining problem is to find an appropriate definition of the Fourier coefficients in such a general setting. 
The results could be generalized further if we discover a general  definition of the Fourier coefficients which matches the general local theory by Sugano. 
The method by the result \cite[Theorem 7.1]{B} of Borcherds would  be a strong tool to find such an appropriate definition. In view of the close connection of our lift with the theta lift, we may remark that our results would be understood totally in a representation theoretic manner, e.g.\ some detailed study on the Weil representation~(Howe correspondence or the theta correspondence). In fact, as we have remarked, the cuspidal representation that $F_f$ generates admits a special Bessel model at non-archimedean places and this fact is compatible with the characterization of the theta lifts by Piatetskii-Shapiro and Soudry \cite{Ps-S} in terms of the special Bessel model. 
In addition, we recall that the lifts in \cite{Mu-N-P} correspond to some residual automorphic forms on $GL(4)$ via the Jacquet-Langlands correspondence. Then a natural question is ``what are the automorphic forms or reprsentations~(maybe non-cuspidal) on the split orthogonal group corresponding to our lifts via the Jacquet-Langlands correspondence?''. 
Now it seems that this problem has been becoming more accessible than before due to the recent advancement of the Arthur trace formula \cite{A}~({see also \cite{At-G}}).
Finally we however remark that we provide the explicit construction of $F_f$ so that it is accessible to those without knowledge of the representation theory, e.g. those studying automorphic forms in the classical setting. In fact, though there is a well-known approach to show the non-vanishing of theta lifts by Rallis inner product formula~(cf.~\cite{Ra},~\cite{Ga-Q-Ta}), the non-vanishing of our lift $F_f$ is proved by the representability of every positive integer in the even unimodular lattice $E_8$ and some elementary argument of the $L$-function of $f$, the latter of which is also used in \cite[Section 4]{Mu-N-P}.  
We hope that such accessibility leads to a broader development of studies on automorphic forms. 

Let us explain the outline of the paper. In Section \ref{Notation} we introduce basic notations of algebraic groups and Lie algebras, and automorphic forms necessary for later argument. 
In Section \ref{Explicit-lift} we introduce an automorphic form $F_f$ by a lifting from a Maass cusp form $f$ of level one. We first define it as an automorphic form on $O(1,8n+1)(\R)$. 
We prove that $F_f$~(on $O(1,8n+1)(\R)$) is a theta lift from $f$, which implies the left $\Gamma_S$-invariance of $F_f$. We next adelize $F_f$. We then verify that $F_f$ is a cusp form and show the non-vanishing of $F_f$. 
In Section \ref{Hecke-theory} we develop the Hecke theory for $F_f$ and derive the simple expression for Hecke eigenvalues of $F_f$, and in Section \ref{Cusp-rep} we study the cuspidal representation $\pi_{F_f}$ generated by $F_f$ in detail. 
We determine all of its local component explicitly. That enables us to discuss its non-temperedness at finite places and have the explicit formula for the standard $L$-function of $F_f$ or $\pi_{F_f}$. As we have remarked, we need Sugano's local theory \cite[Section 7]{Su} to study the cuspidal representations and the standard $L$-functions for our lifts. 
In the appendix we have similar results on cuspidal representations and the standard $L$-functions for the lifting by Oda \cite{Od} and Rallis-Schiffman \cite{Ra-Sch}, to which Sugano's local theory was originally applied. As is expected, such cuspidal representations are proved to be non-tempered at finite places.
\section{Basic notations}\label{Notation}
\subsection{Algebraic groups}\label{gps-hypsp}
For $N \in \mathbb{N}$, let $S\in M_N(\Q)$ be a positive definite symmetric matrix and put $Q:=
\begin{pmatrix}
& & 1\\
& -S & \\
1 & &
\end{pmatrix}$. 
We then define a $\Q$-algebraic group $\cG$ by the group
\[
\cG(\Q):=\{g\in M_{N+2}(\Q)\mid {}^tgQg=Q\}
\]
of $\Q$-rational points. We introduce another $\Q$-algebraic group $\cH$ by the group 
\[
\cH(\Q):=\{h\in M_N(\Q)\mid {}^thSh=S\}
\]
of $\Q$-rational points.
Let $q_S$, resp.\ $q_Q$, denote the quadratic form on $\Q^N$, resp.\ $\Q^{N+2}$, associated to $S$, resp.\ $Q$, i.e.
$$
q_S(v) = \frac{1}{2}{}^tv S v,~
q_Q(w) = \frac{1}{2}{}^tw Q w
$$
for $v \in \Q^N$ and $w \in \Q^{N+2}$.
Then $\cH$, resp.\ $\cG$, is the orthogonal group associated to this quadratic form.
For every place $v\leq\infty$ of $\Q$ we put $G_v:=\cG(\Q_v)$ and  $H_v:=\cH(\Q_v)$.

In addition, we introduce the standard proper $\Q$-parabolic subgroup $\cP$ with the Levi decomposition $\cP=\cN\cL$, 
where the $\Q$-subgroups $\cN$ and $\cL$ are defined by
\begin{align*}
\cN(\Q)&:=\left\{\left.n(x)=
\begin{pmatrix}
1 & {}^txS & \frac{1}{2}{}^txSx\\
  & 1_N & x\\
  &      & 1
\end{pmatrix}~\right|~x\in\Q^N\right\}, \\
\cL(\Q)&:=\left\{\left.
\begin{pmatrix}
\alpha & & \\
  & \delta & \\
  &          & \alpha^{-1}
\end{pmatrix}~\right|~\alpha\in\Q^{\times},~\delta\in\cH(\Q)\right\}.  
\end{align*}

Let $J = \Z^2$ be the hyperbolic plane, $L$ be a maximal lattice with respect to $S$, and put
\[
L_0:=\left\{\left.\begin{pmatrix}
x\\
y\\
z
\end{pmatrix}\in\Q^{N+2}\right|~x,z\in\Z,~y\in L\right\} = L \oplus J,
\]
which is a maximal lattice with respect to $Q$. Here see \cite[Chapter II,~Section 6.1]{Shi} for the definition of maximal lattices.
Through the bilinear form induced by the quadratic form $q_S$, the dual lattice $L^\sharp := \Hom_\Z(L, \Z)$ is identified with a sublattice of $\Q^N$ containing $L$, and is  maximal with respect to $S$ if and only if $L$ is even unimodular.
For each finite prime $p<\infty$ we introduce $L_{0,p}:=L_0\otimes_{\Z}\Z_p$ and put
\[
K_p:=\{g\in G_p\mid gL_{0,p}=L_{0,p}\},
\]
which forms a maximal open compact subgroup of $G_p$. 
On the other hand, let $R:=
\begin{pmatrix}
1 & & \\
& S & \\
& & 1
\end{pmatrix}$ and put
\[
K_{\infty}:=\{g\in G_{\infty}\mid {}^tgRg=R\}, 
\]
which is a maximal compact subgroup of $G_{\infty}$. 
Let $K_f:=\prod_{p<\infty}K_p$ and $K:=K_f\times K_{\infty}$. The groups $K_f$ and $K$ form maximal compact subgroups of $\cG(\A_f)$ and $\cG(\A)$ respectively.
We furthermore put $U:=U_f\times H_{\infty}$ with $U_f:=\prod_{p<\infty}U_p$, where
\[
U_p:=\{h\in\cH(\Q_p)\mid hL_p=L_p\}
\]
with $L_p:=L\otimes_{\Z}\Z_p$. We now set
\begin{equation}
  \label{eq:GammaS}
\Gamma_S:=\cG(\Q)\cap K_f G_{\infty} = \{\gamma\in\cG(\Q)\mid \gamma L_0=L_0\}. 
\end{equation}
and have the following result.
\begin{lem}\label{Classnum-Cusps}
(1)~(Strong approximation theorem for $\cG$)~The class number of $\cG=O(Q)$ with respect to $G_{\infty}K_f$ is one. Namely $\cG(\A)=\cG(\Q)G_{\infty}K_f$\\
(2)~The class number of $\cH=O(S)$ with respect to $U$ coincides  with the number of $\Gamma_S$-cusps.
\end{lem}
\begin{proof}
For (1) see \cite[Lemma 9.23 (i),~Theorem 9.26]{Shi}, for which note that the base field is  $\Q$ in our case. As for the second assertion we can verify that the number of the cusps coincides with the class number of the Levi subgroup $\cH$, following the proof of \cite[Lemma 2.3, 2]{Mu-N-P}. More specifically, in view of the strong approximation theorem, we have the bijection
\[
\cP(\Q)\backslash\cG(\Q)/\Gamma_S\ni\cP(\Q)\gamma\Gamma_S
\mapsto\cP(\Q)\gamma G_{\infty}K_f\in\cP(\Q)\backslash\cG(\A)/G_{\infty}K_f,
\]
where $\gamma\in\cG(\Q)$. 
This yields the second assertion. In fact, by virtue of the Iwasawa decomposition $\cG(\A)=\cP(\A)K$, we have a bijection  $\cP(\Q)\backslash\cG(\A)/K_fG_{\infty}\simeq\cH(\Q)\backslash\cH(\A)/U_fU_{\infty}$. 
\end{proof}
\begin{rem}\label{Representability}
  The class number of $\cH$ is also the number of equivalence classes of quadratic forms in the same genus as $S$. Also, there is only one element in the genus of $L$ when $S$ is unimodular. In that case, we see that a $\Gamma_S$-cusp corresponds to a decomposition of $L_0$ into the direct sum of a hyperbolic plane and a negative definite unimodular lattice. 
\end{rem}
For the subsequent argument we remark that the real Lie group $G_{\infty}$ admits an Iwasawa decomposition 
\[
G_{\infty}=N_{\infty}A_{\infty}K_{\infty},
\]
where
\begin{equation}\label{Iwasawa-decomp}
N_{\infty}:=\left\{n({ x})\mid { x}\in\R^N\right\}, \qquad 
A_{\infty}:=\left\{\left.
a_y:=
\begin{pmatrix}
y &  & \\
 & 1_N & \\
 & & y^{-1}
\end{pmatrix}~\right|~y\in\R^{\times}_+\right\}.
\end{equation}
From the Iwasawa decomposition we can identify the homogeneous space $G_{\infty}/K_{\infty}$ with the $(N+1)$-dimensional real hyperbolic space $H_N:=\{(x,y)\mid x\in\R^N,~y\in\R_{>0}\}$ by the natural map
\[
n(x)a_y\mapsto (x,y).
\]
The cusp forms we are going to study are regarded as cusp forms on the real hyperbolic space.
\subsection{Lie algebras}\label{Lie-alg}
The Lie algebra $\g$ of $G_{\infty}$ is defined as
\[
\g=\{X\in M_{N+2}(\R)\mid {}^tXQ+QX=0_{N+2,N+2}\},
\]
which coincides with
\[
\left\{\left.
\begin{pmatrix}
a & {}^tyS & 0\\
x & Y & y\\
0 & {}^txS & -a
\end{pmatrix}~\right|~
\begin{array}{c}
a\in\R,~x,~y\in\R^N\\
Y\in{\mathfrak o}(S)
\end{array}\right\},
\]
where ${\mathfrak o}(S)$ denotes the Lie algebra of $\cH(\R)$. 

Let $\theta$ be the Cartan involution of $\g$ defined by
\[
\g\ni X\rightarrow -R{}^tXR^{-1}\in\g.
\]
We put
\[
\mk:=\{X\in\g\mid \theta(X)=X\},\quad\p:=\{X\in\g\mid \theta(X)=-X\}.
\]
Then a Cartan decomposition $\g=\mk\oplus\p$ is obtained. 
Let $\la$ be a maximal abelian subalgebra given by
\[
\la:=\left\{\left.
\begin{pmatrix}
t & & \\
& 0_{N, N} & \\
& & -t
\end{pmatrix}\right|~t\in\R\right\}.
\]
The algebra $\g$ has an Iwasawa decomposition
\[
\g=\n\oplus\la\oplus\mk,
\]
where
\[
\n:=\left\{\left.
\begin{pmatrix}
0 & {}^t{ x}S & 0\\
0_N & 0_{N, N} & { x}\\
0 & {}^t0_N & 0
\end{pmatrix}~\right|~x\in\R^N\right\}.
\]

We next consider the root space decomposition of $\g$ with respect to $\la$. Let $H:=
\left(\begin{smallmatrix}
1 &  &  \\
 & 0_{N, N} & \\
 & &  -1
\end{smallmatrix}\right)$ and $\alpha$ be the linear form of $\la$ such that $\alpha(H)=1$. 
Then $\{\pm\alpha\}$ is the set of roots for $(\g,\la)$. 
Let $\{{ e}_i\mid 1\leq i\leq N\}$ be an orthonormal basis of the Euclidean space $\R^N$ with respect to $S$. For ${ e}_i$ with $1\le i\le N$ we put
\[
E_{\alpha}^{(i)}:=
\begin{pmatrix}
0 & {}^t{ e}_iS & 0\\
0_N & 0_{N, N} & { e}_i\\
0 & {}^t0_N & 0 
\end{pmatrix},\quad E_{-\alpha}^{(i)}:=
\begin{pmatrix}
0 & {}^t0_N & 0\\
{ e}_i & 0_{N, N} & 0_N\\
0 & {}^t{ e}_iS & 0 
\end{pmatrix}.
\]
The set $\{E_{\alpha}^{(i)}\mid 1\le i\le N\}$~
(respectively~$\{E_{-\alpha}^{(i)}\mid 1\le i\le N\}$) forms a basis of $\n$~(respectively~a basis of $\bar{\n}:=\left\{\left.
\left(\begin{smallmatrix}
0 & {}^t{ 0_N} & 0\\
S{ x} & 0_{N,N} & 0_N\\
0 & {}^t{ x} & 0
\end{smallmatrix} \right)\right|~x\in\R^N\right\}$). Let ${\mathfrak z}_{\la}(\mk):=\{X\in\mk\mid [X,A]=0~\forall A\in\la\}$, which coincides with
\[
\left\{\left.
\begin{pmatrix}
0 & {}^t0_N & 0\\
0_N & Y & 0_N\\
0 & {}^t0_N & 0
\end{pmatrix}~\right|~Y\in{\mathfrak {so}}(S)\right\}.
\]
Then ${\mathfrak z}_{\la}(\mk)\oplus\la$ is the eigen-space with the eigenvalue zero. We then see from the root space decomposition of $\g$ with respect to $\la$ that $\g$ decomposes into 
\[
\g=({\mathfrak z}_{\la}(\mk)\oplus\la)\oplus\n\oplus\bar{\n}.
\]

We now introduce the differential operator $\Omega$ defined by the infinitesimal action of 
\begin{equation}\label{Casimir-defn}
\Omega := \frac{1}{2N}H^2-\frac{1}{2}H+\frac{1}{N}\sum_{i=1}^{N}{E_{\alpha}^{(i)}}^2.
\end{equation} 
This differential operator $\Omega$ coincides with the infinitesimal action of the Casimir element of $\g$ (see \cite[p.293]{Kn}) on the space of right $K$-invariant smooth functions of $G_{\infty}$. To check this we note $[E_{\alpha}^{(i)},E_{-\alpha}^{(\bar{i})}]=H$ for $z\in\Hl^1$ and $E_{\alpha}^{(i)}-E_{-\alpha}^{(i)}\in\mk$. 
In what follows, we call $\Omega$ the Casimir operator.
\subsection{Automorphic forms}\label{Autom-form}
For $\lambda\in\C$ and a congruence subgroup $\Gamma\subset \SL_2(\R)$ we denote by $S(\Gamma,\lambda)$ the space of Maass cusp forms of weight $0$ on the complex upper half plane ${\mathfrak h}:=\{u+\sqrt{-1}v\in\C\mid v>0\}$ whose eigenvalue with respect to the hyperbolic Laplacian is $-\lambda$.

We continue with the same notations from the previous two sections.
 For $r\in\C$ we denote by $\cM(\Gamma_S,r)$ the space of smooth functions $F$ on $G_{\infty}$ satisfying the following conditions: 
\begin{enumerate}
\item $\Omega\cdot F=\displaystyle\frac{1}{2N}\left(r^2-\displaystyle\frac{N^2}{4}\right)F$, where $\Omega$ is the Casimir operator defined at \eqref{Casimir-defn},
\item for any $(\gamma,g,k)\in\Gamma_S\times G_{\infty}\times K_{\infty}$, we have $F(\gamma gk)=F(g)$,
\item $F$ is of moderate growth.
\end{enumerate}
{As usual} we can say $F \in \cM(\Gamma_S,r)$ is a cusp form if it vanishes at all the cusps of $\Gamma_S$. {We will explain this in Section \ref{Adelic-lift} with the help of the adelic language. Though we assume there that $L$ is even unimodular, the explanation works for any maximal lattice $L$}. 

Let $K_{\alpha}$, with $\alpha\in\C$, denote the modified Bessel function (see \cite[Section 4.12]{A-A-R}), which satisfies the differential equation
\[
y^2\frac{d^2 K_{\alpha}}{dy^2}+y\frac{dK_{\alpha}}{dy}-(y^2+{\alpha}^2)K_{\alpha}=0.
\]
With $K_{\alpha}$ we can describe the Fourier expansion of $F$ as follows:
\begin{prop}\label{four-exp-prop}
Let $L^{\sharp}$ be the dual lattice of $L$. An automorphic form $F\in\cM(\Gamma_S,r)$ admits a Fourier expansion
\[
F(n(x)a_y)=\sum_{\lambda\in L^{\sharp}}W_{\lambda}(a_y)\exp(2\pi\sqrt{-1}({}^t\lambda S{ x})).
\]
Here
\[
W_{\lambda}(a_y)=
\begin{cases}
  C_{\lambda}y^{N/2}K_r\left(4\pi y\sqrt{ q_S(\lambda)}\right)&(\lambda\not=0_N),\\
c_1y^{N/2-r}+c_2y^{N/2+r}&(\lambda=0_N,~r\not=0),\\
c_1y^{N/2}+c_2y^{N/2}\log(y)&(\lambda=0_N,~r=0),
\end{cases}
\]
where $C_{\lambda}$, $c_1$ and $c_2$ are constants.
\end{prop}
\begin{proof}
The condition $\Omega\cdot F=\displaystyle\frac{1}{2N}(r^2-\frac{N^2}{4})F$ implies that $W_{\lambda}$ satisfies the same condition. We note that $W_{\lambda}$ is determined by its restriction to $A$. For simplicity of the notation we put $W_{\lambda}(y):=W_{\lambda}(a_y)$ and $\hat{W}_{\lambda}(y):=y^{-\frac{N-1}{2}}W_{\lambda}(y)$. 

We then verify that $\hat{W}_{\lambda}(y)$ satisfies the differential equation
\[
  \frac{\partial^2}{\partial y^2}\hat{W}_{\lambda}(y)-\left((8\pi^2)
    2q_S(\lambda)
    +\frac{r^2-\frac{1}{4}}{y^2}\right)\hat{W}_{\lambda}(y)=0.
\]
When $\lambda=0_N$ it is easy to show that we have
\[
W_{\lambda}(y)=
\begin{cases}
c_1y^{N/2+r}+c_2y^{N/2-r}&(\lambda=0_N,~r\not=0),\\
c_1y^{N/2}+c_2y^{N/2}\log(y)&(\lambda=0_N,~r=0).
\end{cases}
\] 
Now assume $\lambda\not=0_N$. 
Putting $Y:=8\pi y\sqrt{q_S(\lambda)}$, the differential equation above is rewritten as
\[
(\frac{\partial^2}{\partial Y^2}+(-\frac{1}{4}+\frac{\frac{1}{4}-r^2}{Y^2}))\hat{W}_{\lambda}\left(\frac{Y}{8\pi\sqrt{q_S(\lambda)}}\right)=0.
\]
This is precisely the differential equation for the Whittaker function (see \cite[Section 4.3]{A-A-R}). With the Whittaker function $W_{0,r}$ parametrized by $(0,r)$ we thereby have the moderate growth solution
\[
  \hat{W}_{\lambda}(y)=C_{\lambda}W_{0,r}\left(8\pi y \sqrt{q_S(\lambda)}\right).
\]
with constants $C_{\lambda}$ depending only on $\lambda$.
We now note the relation (see \cite[Section 13, 13.18 (iii)]{O-L-B-C})
\[
W_{0,r}(2y)=\sqrt{\frac{2y}{\pi}}K_{r}(y).
\]
This means that $F$ has the Fourier expansion as in the statement of the proposition.
\end{proof}
\section{An explicit lifting construction}\label{Explicit-lift}
Let us now assume that $L$ is an even unimodular lattice. 
We can identify $L$ with the quadratic module $(\Z^N, q_S)$, where $S$ is a positive definite symmetric matrix satisfying $S^{-1}\Z^N=\Z^N$ and $N = 8n$ for some $n \in \mathbb{N}$. 
Then the dual lattice $L^{\sharp}=S^{-1} \Z^N$ is the same as $L$.

Let ${\mathfrak h}$ denote the complex upper half plane and let $$f(\tau)=\sum_{n\not=0}c(n)W_{0,\frac{\sqrt{-1}r}{2}}(4\pi|n|v)\exp(2\pi\sqrt{-1}nu)\in S\left(SL_2(\Z);-\frac{r^2 + 1}{4}\right)$$ be a Maass cusp form on ${\mathfrak h}$, where note that the Selberg conjecture means that $r$ is a real number.
For $\lambda\in\R^n$ we put $|\lambda|_S:=\sqrt{q_S(\lambda)}$.
To consider the lifting for an even unimodular lattice $L$, we define the Fourier coefficient $A(\lambda)$ for a non-zero $\lambda\in L^{\sharp}=L$ as in \eqref{Fouriercoeff-lifting} by
$$
A(\lambda):=|\lambda|_S\sum_{d|d_{\lambda}}c\left(-\frac{|\lambda|_S^2}{d^2}\right)d^{4n-2},
$$
where $d_{\lambda}$ denotes the greatest common divisor of the non-zero entries of $\lambda$. The lifting from $f$ to $F_f$ on $G_{\infty}$ is defined by
\[
F_f(n(x)a_y)=\sum_{\lambda\in L^{\sharp}\setminus\{0\}}A(\lambda)y^{4n}K_{\sqrt{-1}r}(4\pi|\lambda|_Sy)\exp(2\pi\sqrt{-1}{}^t\lambda Sx).
\]
The aim of this section is to prove the following result:
\begin{thm}\label{Construction-lifting}
The automorphic form $F_f$ is a cusp form in $M(\Gamma_S,\sqrt{-1}r)$. 
Furthermore, if $f$ is non-zero, so is $F_f$.
\end{thm}
As mentioned in the introduction, we obtain the $\Gamma_S$-invariance of $F_f$ by recognizing it as a theta lift from modular forms on ${\mathfrak h}$. For this, we need to calculate the Fourier expansion of the theta lift explicitly using \cite[Theorem 7.1]{B} and check that it agrees with the Fourier expansion of $F_f$. Recognizing $F_f$ as a theta lift does not give us cuspidality or non-vanishing directly. The former follows from reinterpreting $F_f$ in the adelic setting, whereas the latter is a consequence of representability of integers by the $E_8$ lattice and properties of the $L$-function of $f$. 
\subsection{Real hyperbolic space as a Grassmanian manifold}\label{Hyp-Grassmanian}
For this section, we let $(\R^N, q_S)$ be any positive definite, real quadratic space, where $q_S$ denotes a quadratic form defined by an arbitrary positive symmetric matrix $S$.  
To use the result of Borcherds mentioned above we need an identification of the real hyperbolic space $H_N$ with the Grassmanian ${\cal D}$ of positive oriented lines in the quadratic space $V_N:=(\R^{N+2},q_Q)$, where $q_Q(a, x, b) := ab - q_S(x)$. Let $B_Q$ denote the bilinear form associated with $q_Q$. 
For every $(x,y)\in H_N$, we have a vector
\begin{equation}
  \label{eq:nu}
\nu(x,y):=\frac{1}{\sqrt{2}}{}^t(y+y^{-1}q_S(x) ,- y^{-1}x,y^{-1}) \in V_N
\end{equation}
satisfying $B_Q(\nu(x,y), \nu(x,y))=1$. It generates the positive, oriented line $\R\cdot\nu(x,y)$, which is an element in ${\cal D}$. In fact, we see that ${\cal D}=\{\R\cdot\nu(x,y)\mid (x,y)\in H_N\}$.

We now note that the quadratic space $V_N$ is isometric to $\R^{1,N+1}$, where $\R^{p,q}$ denotes the real vector space $\R^{p+q}$ with the quadratic form
\[
Q_{p,q}(x_1,x_2,\cdots,x_{p+q}):= 
\frac{1}{2}\left(\sum_{i=1}^px_i^2-\sum_{j=1}^qx_{p+j}^2\right).
\]
We slightly abuse the notation by using $\nu$ to represent the line generated by $\nu(x,y)$. Every line $\nu\in{\cal D}$ induces an isometry 
\begin{align*}
\iota_{\nu}:{V_N} & \to {\R \cdot \nu \oplus(\nu^{\perp},{q_S|_{{\nu}^{\perp}}})\simeq \R^{1,N+1}}\\
\lambda&\mapsto (\iota^+_\nu(\lambda), \iota^-_\nu(\lambda)),
\end{align*}
where 
  $$
 \iota_\nu^+( \lambda):=B_Q(\lambda,\nu)\nu,~
 \iota_\nu^-( \lambda):= \lambda -  \iota_\nu^+( \lambda) \in \nu^\perp
$$ 
are the components of $\lambda$. Note that $\iota^+_{\gamma \cdot \nu}(\gamma \cdot \lambda) = \gamma \cdot \iota^+_\nu( \lambda)$ for any $\gamma \in \mathcal{G}(\R)$ and $\lambda \in V_N$.

\subsection{Theta lift from $f$ and its coincidence with $F_f$}\label{Theta-lift}
We now resume the assumption that $L$ is even unimodular, thus $N = 8n$ for some $n \in \mathbb{N}$. We denote it also by $(L,q_S)$ and its direct sum with a hyperbolic plane by $(L_0, q_Q)$ as in section \ref{gps-hypsp}. Since $L$ is even unimodular, so is $L_0$.
Let ${\cal D}$ be the Grassmannian associated to $V_N = L_0 \otimes_\Z \R$.

To introduce the theta lift of a Maass cusp form $f$ of level one, let $P_n(x):= 2^{-2n - 3} x^{4n}$ a polynomial on $\R^{}$. 
The operator $\exp(-\partial_x^2/(8\pi v))$ acts on $P_n$ and produces a polynomial $\mathcal{P}_{n, v}$ on $\R^{}$, which is closely related to the physicists' Hermite polynomial by a change of variable. 
Furthermore, we can consider  $\mathcal{P}_{n, v}$ as a polynomial on $\R^{1, 8n + 1}$ after precomposing it with the surjection $\R^{1, 8n+1}\rightarrow\R$ that sends $(x_1, ..., x_{8n + 2})$ to $x_1$. 
Now, we define the theta series $\Theta_{L_0}$ on ${\mathfrak h}\times {\cal D}$ by the following absolutely convergent sum
\[
\Theta_{L_0}(\tau,\nu):=\sum_{\lambda\in L_0} \mathcal{P}_{n, v}(B_Q(\iota_\nu(\lambda), \nu))\exp(2 \pi\sqrt{-1}
(q_Q(\iota^+_\nu(\lambda))\tau + q_Q(\iota^-_\nu(\lambda)\bar{\tau})).
\]
Comparing to \cite[Section 4]{B}, it is easy to see that this is the theta kernel used by Borcherds with a homogeneous polynomial of degrees $(4n, 0)$. 
By its definition and the Poisson summation formula, we know that $v^{4n+1/2} \Theta_{L_0}(\tau, \nu)$ is a modular function in $\tau \in {\mathfrak h}$ with respect to $SL_2(\Z)$ for any $\nu \in {\cal D}$ (see e.g.\ \cite[Theorem 4.1]{B}).
Then the theta lift $\Phi_{L_0}(\nu,f)$ of the Maass cusp form $f$ is defined as
\[
\Phi_{L_0}(\nu,f):=\displaystyle\int_{SL_2(\Z)\backslash{\mathfrak h}}f(\tau)\overline{\Theta_{L_0}(\tau,\nu)}v^{4n+\frac{1}{2}}d\mu(\tau),
\]
where $d\mu(\tau):=v^{-2}dudv$.
Since $f$ is a cusp form, the integral above converges absolutely.
On the other hand, since
$$
B_Q(\iota_{\gamma \cdot \nu}(\lambda), \gamma \cdot \nu) = 
B_Q(\iota^+_{\gamma \cdot \nu}(\lambda), \gamma \cdot \nu) = 
B_Q(\iota^+_{\nu}(\gamma^{-1} \lambda), \nu) = 
B_Q(\iota_{\nu}(\gamma^{-1} \lambda), \nu)
$$
for any $\lambda \in V_{8n}$ and $\gamma \in \mathcal{G}(\R)$, we have $\Theta_{L_0}(\tau, \gamma\cdot \nu) = \Theta_{L_0}(\tau,  \nu)$ for any $\gamma \in \Gamma_S$ and $\tau \in \mathfrak{h}$. The integral $\Phi_{L_0}(\nu, f)$ is also left $\Gamma_S$-invariant.

If we choose another unimodular lattice $(L', q_{S'})$, then $L_0' := L' \oplus J$ is isomorphic to $L_0$ by the classification of indefinite unimodular lattices \cite[Chapitre V,~Section 2.2, Theorem 6]{Se}.
Different choices of such decomposition correspond to different cusps of the hyperbolic space $H_N$~(cf.~Remark \ref{Representability}), each of which gives the coordinates to express the Fourier expansion.
To do this, we follow \cite[Theorem 7.1]{B} and choose the isotropic vectors
\[
l:={}^t(1, 0,0),~\Check{l}:={}^t(0,0, 1)
\]
in $L_0$. We then have
\begin{align*}
l_{\nu}&:= \iota^+_\nu(l) = \frac{1}{\sqrt{2}y}\nu,~l_{\nu^{\perp}}:=l-l_{\nu},~B_Q(l_{\nu},l_{\nu})=-B_Q(l_{\nu^{\perp}},l_{\nu^{\perp}})=\frac{1}{2y^2},\\
\mu&=-\Check{l}+\frac{1}{2B_Q(l_{\nu},l_{\nu})}l_{\nu}+\frac{1}{{2B_Q(l_{\nu^{\perp}},l_{\nu^{\perp}})}}l_{\nu^{\perp}}=-\Check{l}+y^2(2l_{\nu}-l),
\end{align*}
and
\[
P_n(B_Q(\iota_{\nu}(\lambda), \nu))=
2^{-2n-3}B_Q(\lambda,\nu)^{4n}=2^{-3}y^{4n}B_Q(\lambda,l_{\nu})^{4n}.
\]
We furthermore note that the orthogonal complement of $l_{\nu^{\perp}}$ in $\nu^{\perp}$ is $L\otimes\R$, which means
$$
B_Q(\lambda, \mu) = 2y^2 B_Q(\lambda, l_\nu) = \sqrt{2}y B_Q(\lambda, \nu) = {}^t\lambda S x
$$
for any $\lambda \in L\otimes \R \subset V_{8n}$.
With the input datum above, we can apply \cite[Theorem 7.1]{B} to write out the Fourier expansion of $\Phi(\nu,f)$.
In the notation loc.\ cit., we have $M = L_0 = L^\sharp_0 = M'$, $K' = L^\sharp = L = K$, $\lambda_{w^+} = 0$, $\lambda_{w^-} = \lambda$,
$$
z_{v^+} = l_\nu,~
|z_{v^+}| = \sqrt{z_{v^+}^2} = \sqrt{2q_Q(l_\nu)},~
p_{w, h^+, h^-}
\begin{cases}
y^{4n}/8, &(h^+, h^-) = (4n, 0), \\
0, & \text{ otherwise.}
\end{cases}
$$
Therefore, the term involving $\Phi_K(w, p_{w, h, h}, F_K)$ vanishes identically and $(-\Delta)^j \overline{p}_{w, h^+, h^-}$ is identically zero for $j \ge 1$.
Since $L$ is unimodular, the term $\sum_{\delta \in M'/M, \delta\mid L = \lambda} \mathbf{e}(n(\delta, z'))$ in the third line of Borcherds' Theorem 7.1 becomes the factor 1.
Furthermore, the coefficient $c_{\delta, \lambda^2/2}(y)$ is $c(-q_S(\lambda)) W_{0, \sqrt{-1}r/2}(4\pi q_S(\lambda) v)v^{4n+1/2} $, which is the Fourier coefficient of $v^{4n+1/2} f(\tau)$.
The extra term of $v^{4n+1/2}$ comes from the way Borcherds normalized his input (which has an extra factor of $v^{m^+ + b^+/2}$). 
With these in hand and that $c(0) = 0$, Theorem 7.1 of \cite{B} simplifies to
\begin{align*}
\Phi_{L_0}(\nu(x, y), f) &=
\frac{y^{4n +1}}{4} \sum_{\lambda\in L^\sharp \backslash\{0\}}c(-|\lambda|_S^2) \sum_{m\ge 1}m^{4n}\exp(2\pi \sqrt{-1}m {}^t\lambda Sx) \displaystyle\\ 
&\quad \int_0^{\infty}W_{0,\sqrt{-1}r/2}(4\pi|\lambda|_S^2v)e^{-\frac{\pi m^2y^2}{v}-2\pi v|\lambda|_S^2}v^{-{2}}dv\\
&=y^{4n}\sum_{\lambda\in L^{\sharp} \backslash\{0\}}c(-|\lambda|_S^2)\sum_{m\ge 1}m^{n/2-1}\exp(2\pi\sqrt{-1}m{}^t\lambda Sx)|\lambda|_S^{}K_{\sqrt{-1}r}(4\pi m |\lambda|_Sy)\\
&= y^{4n}\sum_{\lambda\in L^{\sharp}\setminus\{0\}}
A(\lambda) K_{\sqrt{-1}r}(4\pi|\lambda|_Sy)\exp(2\pi\sqrt{-1}{}^t\lambda Sx),
\end{align*}
where $A(\lambda)$ is defined in \eqref{Fouriercoeff-lifting}.
Here, for the second equation, we have made the change of variable $v \mapsto 1/v$ and used the following integral identity
\[
\displaystyle\int_0^{\infty}\exp(-pt-a/(2t))W_{0,\sqrt{-1}r/2}(a/t)dt=2\sqrt{a/p}K_{\sqrt{-1}r}(2\sqrt{ap})
\]
(cf.~\cite[4.22 (22)]{EMOT})~with $a=4\pi|\lambda|^2_S$ and $p=\pi m^2y^2$. 
We now immediately see that $F_f$ coincides with $\Phi_{L_0}(\nu,f)$, which is left $\Gamma_S$-invariant.
\subsection{Adelic formulation of the lifting and the proof of Theorem \ref{Construction-lifting}}\label{Adelic-lift}
We reformulate the lifting $F_f$ in the adelic setting to complete the proof of Theorem \ref{Construction-lifting}. 

For this purpose we introduce the special orthogonal group $SO(S):={\cal H}\cap SL_N$ over $\Q$, where $SL_N$ denotes the $\Q$-algebraic group defined by the special linear group of degree $N$. It is easy to verify that the cosets ${\cal H}(\Q)\backslash{\cal H}(\A)/U_fU_{\infty}$ have representatives in $SO(S)(\A)$. 

We shall recall that the cosets $SO(S)(\Q)\backslash SO(S)(\A)/(SO(S)(\A)\cap U_fU_{\infty})$ are in bijection with the equivalence classes of the quadratic forms of the same genus with $S$~(cf.~\cite[ChapterII,~Section 9.25]{Shi}). 
To describe this bijection we recall that each $h\in SO(S)(\A)$ has a decomposition $h=au^{-1}$ with $(a,u)\in SL_N(\Q)\times(\prod_{p<\infty}SL_N(\Z_p)\times SL_N(\R))$ by the strong approximation theorem of $SL_N$. Then $S_h:={}^taSa={}^tuSu$ is in the same genus with $S$.  
The bijection above is induced by the mapping 
\[
h\mapsto S_h.
\]
Furthermore we remark that, if $h\in\cH(\A)\setminus SO(S)(\A)$, there exists $\delta_0 \in \cH(\Q)\setminus SO(S)(\Q)$ such that $\delta_0 h\in SO(S)(\A)$. 
We can thus say that $h\in\cH(\A)$ has a decomposition $h=au^{-1}$ with $(a,u)\in GL_N(\Q)\times(\prod_{p<\infty}SL_N(\Z_p)\times SL_N(\R))$ in general.  
We put $L_h:=(\prod_{p<\infty} h_p\Z_p^N\times\R^N)\cap\Q^N$ for $h=(h_v)_{v\le\infty}\in\cH(\A)$. 
Then we have $L_h=a\Z^N$.

We then see that ${\cal H}(\Q)\backslash{\cal H}(A)/U_fU_{\infty}$ can be viewed as a subset of the double coset space $SO(S)(\Q)\backslash SO(S)(\A)/SO(S)(\A)\cap U_fU_{\infty}$. 
We let $C(S):=\{S_1,~S_2,\ldots,S_c\}$ be the classes of the same genus with $S$ corresponding bijectively to a complete set $\{h_i\in SO(S)(\A_f) \mid 1\le i\le c\}$ of representatives for ${\cal H}(\Q)\backslash{\cal H}(\A)/U_fU_{\infty}$~(which are called {the $O$-classes in \cite[Chapter II,~Section 9.27]{Shi}}).

Let $f \in  S(SL_2(\Z);-(\frac{1}{4}+\frac{r^2}{4}))$ be a Maass cusp form with Fourier expansion $f(z)=\sum_{n\not=0}c(n)W_{0,\frac{\sqrt{-1}r}{2}}(4\pi|n|y)\exp(2\pi\sqrt{-1}nx)$. Let $\Lambda$ be the standard additive character of $\A/\Q$. We introduce the Fourier series
\begin{equation}\label{F-series}
F_f(n(x)a_ykg):=\sum_{\lambda\in\Q^n\setminus\{0\}}F_{f,\lambda}(n(x)a_ykg)~
\quad\forall(x,y,k,g)\in\A^n\times\R^{\times}_{+}\times K_{\infty}\times\cG(\A_f)
\end{equation}
with
\[
F_{f,\lambda}(n(x)a_ykg):=A_{\lambda}(g)y^{4n}K_{\sqrt{-1}r}(4\pi|\lambda|_Sy)\Lambda({}^t\lambda S x),
\]
where $A_{\lambda}(g)$ is defined by the following three conditions:
\begin{align}
A_{\lambda}\left(
\begin{pmatrix}
1 & & \\
  & h & \\
  &   & 1
\end{pmatrix}\right)&:=
\begin{cases}
|\lambda|_S\sum_{d|d_{\lambda}}c(-\displaystyle\frac{|\lambda|_S^2}{d^2})d^{4n-2}&(\lambda\in L_h)\\
0&(\lambda\in\Q^N\setminus L_h)
\end{cases},\label{Def-cond-1}\\
A_{\lambda}\left(
\begin{pmatrix}
\beta & & \\
         & h & \\
         &               & \beta^{-1}
\end{pmatrix}\right)
&:=||\beta||_{\A}^{4n}A_{||\beta||_{\A}^{-1}\lambda}\left(
\begin{pmatrix}
1 & & \\
  & h & \\
  &     & 1
\end{pmatrix}\right),\label{Def-cond-2}\\
A_{\lambda}(n(x)gk)&:=\Lambda({}^t\lambda Sx)A_{\lambda}(g)~\forall (x,g,k)\in\A_f^N\times\cG(\A_f)\times K_f\label{Def-cond-3}.
\end{align}
Here
\begin{enumerate}
\item $(\beta,h)\in\A_f^{\times}\times\cH(\A_f)$ and $||\beta||_{\A}$ denotes the idele norm of $\beta$,
\item $d_{\lambda}$ denotes the greatest common divisor of the non-zero entries in $a^{-1}\lambda\in\Z^N~(=a^{-1}L_h)$.
\end{enumerate}
Let us remark that we are using the same notation $F_f$ for both the adelic lift above and the non-adelic lift of the previous section. This should not lead to any confusion since it will be clear from the context which lift we are discussing.

We note that the definition of $d_{\lambda}$ does not depend on the decomposition $h=au^{-1}$~(as we are going to see in the proof of the following lemma).
\begin{lem}\label{Well-def-smfct}
The adelic Fourier series defining the adelic $F_f$ is well-defined and is a left $\cP(\Q)$-invariant and right $K$($=K_fK_{\infty}$)-invariant smooth function on $\cG(\A)$. 
\end{lem}
\begin{proof}
We should first note that $A_{\lambda}$ defines a well-defined function in $h\in\cH(\A_f)$ for each fixed $\lambda\in\Q^N$. 
To see this we have to check that $d_{\lambda}$ does not depend on the decomposition $h=au^{-1}$. If we take another decomposition $h=a'{u'}^{-1}$ with $(a',u')\in GL_N(\Q)\times(\prod_{p<\infty}SL_N(\Z_p)\times SL_N(\R))$, 
we have that $a^{-1}a'=u^{-1}u'\in SL_N(\Z)$ and see that the definition of $d_{\lambda}$ remains the same even if we replace $a$ by $a'$. 
 
By the definition of $A_{\lambda}\left(\left(\begin{smallmatrix}
\beta & & \\
  & h & \\
  &    & \beta^{-1}
\end{smallmatrix}\right)\right)$ in (\ref{Def-cond-1}) and (\ref{Def-cond-2}) we can verify that this is right $(\prod_{p<\infty}\Z_p^{\times})\times U_f(=\cL(\A_f)\cap K_f)$-invariant as a function of $(\beta,h)\in\A_f^{\times}\times\cH(\A_f)(=\cL(\A_f))$ and that 
$\alpha^{\frac{N}{2}}A_{\lambda}\left( \left(\begin{smallmatrix}
\alpha_f & & \\
  & \delta_fh & \\
  &    & \alpha_f^{-1}
\end{smallmatrix}\right)\right)=
A_{\alpha\delta^{-1}\lambda}\left(\left(\begin{smallmatrix}
1 & & \\
  & h & \\
  &    & 1
\end{smallmatrix}\right)\right)$ for $(\alpha,\delta)\in\Q^{\times}_+\times\cH(\Q)$, where $(\alpha_f,\delta_f)$ denotes the finite adele part of $(\alpha,\delta)\in\Q^{\times}_+\times\cH(\Q)$. From the latter we deduce that $F_f$ is left $\cP(\Q)$-invariant and that, as a result of (\ref{Def-cond-3}), $A_{\lambda}$ is well-defined on $\cG(\A_f)$ and right $K_f$-invariant. 
To finish the proof we note that the archimedean part {$y^{4n}K_{\sqrt{-1}r}(4\pi|\lambda|_Sy)\exp(2\pi\sqrt{-1}({}^t\lambda Sx))$} with $(y,x)\in\R_+^{\times}\times\R^N$ is a smooth right $K_{\infty}$-invariant function on $G_{\infty}$. As a result we have seen that {$F_f$} satisfies the desired property in the assertion, and we are thus done.
\end{proof}

For $r\in\C$ we now introduce the space $\cM(\cG,r)$ of smooth functions $F$ on $\cG(\A)$ satisfying the following conditions: 
\begin{enumerate}
\item $\Omega\cdot F=\displaystyle\frac{1}{2N}\left(r^2-\displaystyle\frac{N^2}{4}\right)F$, where $\Omega$ is the Casimir operator defined at \eqref{Casimir-defn},
\item for any $(\gamma,g,k)\in \cG(\Q)\times \cG(\A)\times K$, we have $F(\gamma gk)=F(g)$,
\item $F$ is of moderate growth.
\end{enumerate}
We further remark that $F\in\cM(\cG,r)$ has the Fourier expansion
\[
F(g)=\sum_{\lambda\in\Q^N}F_{\lambda}(g),\quad F_{\lambda}(g):=\displaystyle\int_{\A^N/\Q^N}F(n(x)g)\Lambda(-{}^t\lambda Sx)dx,
\]
where $dx$ is the invariant measure normalized so that the volume of $\A^N/\Q^N$ is one. 
We call $F$ a cusp form if $F_0\equiv 0$ in the Fourier expansion. 

Now let us explain that this cuspidality condition of $F\in\cM(\cG,r)$ implies the vanishing of the constant term of the non-adelic Fourier expansion of $F$ at every $\Gamma_S$-cusp. Let $c\in\cG(\Q)$ be a representative of a $\Gamma_S$-cusp and decompose $c$ into $c_fc_{\infty}$ with $c_f\in\cG(\A_f)$ and $c_{\infty}\in G_{\infty}$. Furthermore let $c$ correspond to a representative $h$ of $\cH(\Q)\backslash\cH(\A)/U_fU_{\infty}$ via the bijection in Lemma \ref{Classnum-Cusps}. From Proposition \ref{four-exp-prop} we then see that the non-adelic Fourier expansion of $F(n(x_{\infty})a_{y_{\infty}}
\left(\begin{smallmatrix}
1 & & \\
 & h & \\
 &  & 1
\end{smallmatrix}\right))$ for $(x_{\infty},y_{\infty})\in\R^n\times\R^{\times}_+$ is expressed as
\[
\sum_{\lambda\in L_h\setminus\{0\}}C_{\lambda}(h)W_{\lambda}(a_{y_{\infty}})\exp(2\pi\sqrt{-1}({}^t\lambda Sx_{\infty}))
\]
with Fourier coefficients $C_{\lambda}(h)$. 
Here note that the above $F$ translated by 
$\left(\begin{smallmatrix}
1 & & \\
 & h & \\
 &  & 1
\end{smallmatrix}\right)$ is right $\left(\begin{smallmatrix}
1 & & \\
 & h & \\
 &  & 1
\end{smallmatrix}\right)K_f\left(\begin{smallmatrix}
1 & & \\
 & h & \\
 &  & 1
\end{smallmatrix}\right)^{-1}$-invariant, which implies that the summation of the expansion runs over $L_h\setminus\{0\}$ with the even unimodular lattice $L_h$. On the other hand, we have that $F(g_{\infty}\left(\begin{smallmatrix}
1 & & \\
 & h & \\
 &  & 1
\end{smallmatrix}\right))=F(c_fg'_{\infty})=F(c_{\infty}^{-1}g'_{\infty})$ with a suitable change of variables $g_{\infty}\rightarrow g'_{\infty}$ in $G_{\infty}$. The Fourier expansion of $F(c^{-1}g'_{\infty})$ is nothing but the non-adelic expansion at a $\Gamma_S$-cusp $c^{-1}$. We therefore see the vanishing of $F$ at every $\Gamma_S$-cusp in the sense of the non-adelic Fourier expansion~(for this see Remark \ref{Representability}). 
\begin{thm}
Given a Maass cusp form $f\in S(SL_2(\Z);-(\frac{1}{4}+\frac{r^2}{4}))$ we define $F_f$ as in (\ref{F-series}). Then $F_f$ is a cusp form in $\cM(\cG,\sqrt{-1}r)$. 
\end{thm}
\begin{proof}  
By the strong approximation theorem for $\cG$~(cf.~Lemma \ref{Classnum-Cusps} (1)) we have an isomorphism $\cM(\cG(\Q),\sqrt{-1}r)\simeq\cM(\Gamma_S,\sqrt{-1}r)$ given by $F\mapsto F|_{\cG(\R)}$. 
 The right $K$-invariance  for the adelized $F_f$ follows immediately from Lemma \ref{Well-def-smfct}. 
We have proved the left $\Gamma_S$-invariance of $F_f|_{\cG(\R)}$ since it coincides with the non-adelic lift. 
By the standard argument in terms of the strong approximation theorem we can deduce that $F_f$ is a left $\cG(\Q)$-invariant function. 
To see that $F_f$ is of moderate growth, note that this is determined by its restriction to $\cG(\R)$. In fact, $F_f|_{\cG(\R)}$ is given by the Fourier series with rapidly decreasing terms and its Fourier coefficients  $A(\lambda)$ satisfy $|A(\lambda)|=O(|\lambda|_S^{\kappa})$ with some $\kappa>0$, which follows from the growth property of the Fourier coefficients $c(n)$ of $f$. From this we verify that $F_f|_{\cG(\R)}$ is at most of polynomial order by estimating the Fourier series of $F_f$. The action of the Casimir $\Omega$ follows from the Fourier expansion of $F_f$. Hence, we get $F_f \in \cM(\cG,\sqrt{-1}r)$. 
From the Fourier series of $F_f$ it is straightforward to see that $F_f$ is cuspidal.
\end{proof}

We are left with discussing the  non-vanishing of $F_f$. To this end we need the following lemma, which is similar to \cite[Lemma 4.5]{Mu-N-P}.
\begin{lem}\label{Maass-form-non-van-coeff}
Let $f\in S(SL_2(\Z);-(\frac{1}{4}+\frac{r^2}{4}))\setminus\{0\}$ with Fourier coefficients $c(m)$. Then, there exist $M > 0, M \in \Z$, such that $c(-M) \neq 0$.
\end{lem}
\begin{proof}
Assume that $c(m) = 0$ for all $m < 0$. Set $f_1(z) = (f(z)+f(-\bar{z}))/2$ and $f_2(z) = (f(z)-f(-\bar{z}))/2$. Then, $f_1, f_2$ are elements of $S(SL_2(\Z);-(\frac{1}{4}+\frac{r^2}{4}))$. In addition, $f_1$ is an even Maass form and $f_2$ is an odd Maass form, with the property that they have the exact same Fourier coefficients corresponding to positive indices. This implies that the $L$-functions for $f_1$ and $f_2$ satisfy $L(s, f_1) = L(s, f_2)$. On the other hand, $L(s,f_1)$ and $L(s,f_2)$ satisfy functional equations with the gamma factors shifted by $1$. If $L(s, f_1) \neq 0$, we obtain an identity of gamma factors, which can be checked to be impossible. This gives us that $f$ has to be zero, a contradiction.
\end{proof}
By Lemma \ref{Maass-form-non-van-coeff}, there is the smallest positive integer $M_0$ such that $c(-M_0)\neq 0$. 
For $n \in \mathbb{N}$, let $E^n_{8}$ be the direct sum of $n$ copies of the $E_8$ lattice. Then there exists {$\lambda_0 \in E^n_{8}$} with norm $M_0$. This follows from the case $n = 1$, which holds since the theta function associated to $E_8$ is the Eisenstein series of weight 4 on $SL_2(\Z)$.
From the Fourier expansion near the cusp determined by $L_0 \cong J \oplus E^{n}_8$~(cf.~Remark \ref{Representability}), we see that 
{$A_{\lambda_0}\not=0$}. Thus $F_f\not\equiv 0$ for a non-zero $f$, which finishes the proof of Theorem \ref{Construction-lifting}. 

In addition, let us note that Weyl's law for $SL_2(\Z)$ by Selberg~(cf.~\cite[Section 11.1]{Iwn}) implies the existence of Maass cusp forms for $SL_2(\Z)$. 
As a result the argument so far implies the following:
\begin{prop}\label{Existence-Maass-Cuspforms}
There exists non-zero $F_f$.
\end{prop}
\section{Hecke theory for the lifting}\label{Hecke-theory}
We are going to discuss the Hecke theory of our lifting. 
In fact, we will show that if $f$ is a Hecke eigenform then so is the lift $F_f$, and we can compute the Hecke eigenvalues of $F_f$ explicitly in terms of those of $f$. 
The method is to use the non-archimedean local theory by Sugano \cite[Section 7]{Su} for the Jacobi form formulation of the Oda-Rallis-Schiffmann lifting \cite{Od},~\cite{Ra-Sch}. 
\subsection{Sugano's local theory}\label{Sugano-local}
In this section we work over a non-archimedean local field $F$ of characteristic not equal to two. 
Let $\varpi$ be a prime element of $F$ and let ${\mathfrak o}$ be the ring of integers in $F$. 
Let $n_0 \leq 4$ and let $S_0\in M_{n_0}(F)$ be an anisotropic even symmetric matrix of degree $n_0$. We introduce the $m \times m$ matrix $J_m:=\left( \begin{smallmatrix}
  &             &1\\
  & \iddots  & \\
1&             &            
\end{smallmatrix}\right)$. 
We denote by $G_m$ the group of $F$-valued points of the  orthogonal group of degree $2m+n_0$ defined by the symmetric matrix  $Q_m:=
\left(\begin{smallmatrix}
 & & J_m\\
& S_0 & \\
J_m & &
\end{smallmatrix}\right)$. When $m=0$ we understand that $Q_m=S_0$. 
In what follows, we suppose that $L_m:={\mathfrak o}^{2m+n_0}$ is a maximal lattice with respect to $Q_m$. Let $K_m$ be the maximal open compact subgroup of $G_m$ defined by the maximal lattice $(L_m,Q_m)$, namely
\[
K_m:=\{g\in G_m\mid gL_m=L_m\}.
\]
We regard $G_i$ with $i\le m$ as a subgroup of $G_m$ by embedding $G_i$ into the middle $(2i+n_0)\times(2i+n_0)$ block of $G_m$. We also regard $K_i$ with $i\le m$ as a subgroup of $K_m$ similarly. 
 
Hereafter we normalize the invariant measure of $G_{m}$ so that the volume of $K_i$ is one for each $i\le m$,  which is justified in view of the existence of the quotient measure for $G_{i+1}/G_i$.

By $\cH_m$ we denote the Hecke algebra for $(G_m,K_m)$. Let $C_m^{(r)}\in\cH_m$ be defined by the double coset $K_mc_m^{(r)}K_m$, where
\[
c_m^{(r)}:=\diag(\varpi,\cdots,\varpi,1,\dots,1,\varpi^{-1},\cdots,\varpi^{-1})\in G_m,
\]
which is the diagonal matrix whose first $r$ entries and last $r$ entries are $\varpi$ and $\varpi^{-1}$ respectively. 
 As is remarked in \cite[Section 7]{Su}, $\{C_m^{(r)}\mid 1\le r\le m\}$ forms generators of the Hecke algebra $\cH_m$. 
Note that the Satake isomorphism also holds for these orthogonal groups although they are not connected~(cf.~\cite[Theorem 5,~Remark 1 after Theorem 3]{Sa}).

Assume that $m\ge 1$ and for $1\le i\le m$, let $n_i(x):={
\left(\begin{smallmatrix}
1 & -{}^txQ_{i-1} & -\frac{1}{2}{}^txQ_{i-1}x\\
   & 1_{2i-2+n_0} & x\\
   &            & 1
\end{smallmatrix}\right)}\in G_{i}$ for $x\in F^{2i-2+n_0}$. 
By $L_{m-1}^{\sharp}$ we denote the dual lattice of $L_{m-1}$ with respect to $Q_{m-1}$. We need the notion that $\lambda\in L_{m-1}^{\sharp}\setminus\{0\}$ is primitive or reduced as follows~(cf.~\cite[p44]{Su}):
\begin{defn}\label{prim-red}
(1)~A vector $\lambda\in L_{m-1}^{\sharp}$ is defined to be primitive if its $(2m-2+n_0)$-th entry is equal to $1$.\\
(2)~A primitive $\lambda\in L_{m-1}$ is called reduced (with respect to $Q_{m-1}$) if 
\[
\begin{pmatrix}
\varpi^{-1} & & \\
         & 1_{2m-4+n_0} &  \\
         &            &\varpi
\end{pmatrix}n_{m-1}(x)\lambda\not\in \varpi L_{m-1}^{\sharp}
\]
for any $x\in F^{2m-4+n_0}$.
\end{defn}
\begin{lem}\label{reduced}
Suppose that $L_{m-1}^{\sharp}=L_{m-1}$, namely $L_{m-1}$ is self-dual. For a primitive $\lambda\in L_{m-1}^{\sharp}$, $\lambda$ is reduced if and only if the $\varpi$-adic order of $\frac{1}{2}{}^t\lambda Q_{m-1}\lambda$ is not greater than $1$.
\end{lem}
\begin{proof}
We write $\lambda=
\begin{pmatrix}
a\\
\alpha\\
1
\end{pmatrix}$ with $a\in{\mathfrak o}$, $\alpha\in L_{m-2}^{\sharp}$. 
Suppose first that $\lambda$ is reduced. Assume that the $\varpi$-adic order of $\frac{1}{2}{}^t\lambda Q_{m-1}\lambda$ is not less than two, and take $\beta\in F^{2m-4+n_0}$ so that $\alpha+\beta\in \varpi L_{m-2}^{\sharp}$. We then verify that
\[
\begin{pmatrix}
\varpi^{-1} & & \\
   & 1_{2m-4+n_0} & \\
   &             & \varpi
\end{pmatrix}n_{m-1}(\beta)\lambda=
\begin{pmatrix}
\varpi^{-1}(\frac{1}{2}{}^t\lambda Q_{m-2}\lambda-\frac{1}{2}{}^t(\alpha+\beta)Q_{m-2}(\alpha+\beta))\\
\alpha+\beta\\
\varpi
\end{pmatrix}\in\varpi L_{m-1}^{\sharp}.
\]
In fact, noting $\varpi L_{m-2}^{\sharp}=\varpi L_{m-2}$, we see that the $\varpi$-adic order of the first entry for the vector above is not less than one. This contradicts the assumption that $\lambda$ is reduced. 

Suppose next that the $\varpi$-adic order of $\frac{1}{2}{}^t\lambda Q_{m-1}\lambda$ is not greater than one. 
Then, with $\beta\in F^{2m-4+n_0}$ such that $\alpha+\beta\in\varpi L_{m-2}^{\sharp}$, the first entry of $\begin{pmatrix}
\varpi^{-1} & & \\
   & 1_{2m-4+n_0} & \\
   &             & \varpi
\end{pmatrix}n_{m-1}(\beta)\lambda$ has the $\varpi$-adic order is less than or equal to $0$. 
This suffices to prove that $\lambda$ is reduced since the proof is straightforward for $\beta$ such that $\alpha+\beta\not\in\varpi L_{m-2}=\varpi L_{m-2}^{\sharp}$.
\end{proof}

Let us now introduce $M_k:=
\left( \begin{smallmatrix}
\varpi^{-k} & & \\
    & 1_{2m-4+n_0} & \\
    &             & \varpi^k
\end{smallmatrix} \right)\in G_{m-1}$ for a non-negative integer $k$. This is useful to describe elements in $L_{m-1}^{\sharp}$ in terms of reduced ones under the assumption that $L_{m-1}$ is self-dual.
\begin{lem}\label{prim-gcd-red}
Let $L_{m-1}=L_{m-1}^{\sharp}$. Any $\lambda\in L_{m-1}^{\sharp}\setminus\{0\}$ can be written as $u\lambda=\varpi^{k+l}M_k^{-1}\eta$ with some non-negative integers $k$ and $l$, $u\in K_{m-1}$ and a reduced $\eta\in L_{m-1}^{\sharp}$.
\end{lem}
\begin{proof}
Without loss of generality we may assume that $\lambda$ is primitive since we can take $l\ge 0$ and $u\in K_{m-1}$ so that $u\lambda=\varpi^{l}\lambda_0$ with a primitive $\lambda_0$. Note that the case of $l=0$ means that $\lambda$ is primitive, up to the $K_{m-1}$-action on the left.

Let us take $x\in L_{m-2}^{\sharp}$ so that $n_{m-1}(x)\lambda=
\begin{pmatrix}
y\\
0_{2m-4+n_0}\\
1
\end{pmatrix}$ with $y\in{\mathfrak o}$. Note that $n_{m-1}(x)\in K_{m-1}$ for $x\in L_{m-2}^{\sharp}=L_{m-2}$. 
Thus there is $u_0\in K_{m-1},~t\in{\mathfrak o}^{\times}$ and a non-negative integer $f$ such that $u_0\lambda=
\begin{pmatrix}
\varpi^ft\\
0_{2m-4+n_0}\\
1
\end{pmatrix}$. We may now assume $\lambda=
\begin{pmatrix}
\varpi^ft\\
0_{2m-4+n_0}\\
1
\end{pmatrix}$. Put $k:=\left[\frac{f}{2}\right]$, then we have $
\lambda=\varpi^kM_k^{-1}
\begin{pmatrix}
a_0\\
0_{2m-4+n_0}\\
1
\end{pmatrix}$ 
with $a_0\in{\mathfrak o}$ whose $\varpi$-adic order is equal to $0$ or $1$. 
We are therefore done since 
$\begin{pmatrix}
a_0\\
0_{2m-4+n_0}\\
1
\end{pmatrix}$ is reduced in view of Lemma \ref{reduced}.
\end{proof}

We are now ready to introduce the notion of ``local Whittaker functions'' on $G_{m}$ in the sense of \cite[Section 7,~p47]{Su}. Though this does not come from the ``Whittaker model'' in the usual sense, it can be understood in terms of ``special Bessel models'' for unramified principal series  representations of $G_{m}$. 
We furthermore review the notion of the ``local Maass relation'' as in \cite[Section 7,~p52]{Su}, which leads to a nice reduction for the calculation of the Hecke eigenvalues. Sugano's local theory deals with the case of general maximal lattices and our review on this is given in such a general setting. However, what we need is his theory under the assumption ``$\partial=n_0=0$''~(see the notation below and \cite[p6]{Su}). 
We will show that the Fourier coefficients of $F_f$ belong to the space of the local Whittaker functions satisfying the local Maass relation. 

Let $\lambda\in L_{m-1}^{\sharp}$ be reduced and put $H_{\lambda}$ to be the stabilizer of $\lambda$ in $G_{m-1}$. We then introduce the space of the local Whittaker functions as follows: 
\[
\cW_{\lambda}^{\cF}:=\left\{W:G_{m}\rightarrow\C\left|~ \begin{matrix}
W\left(n_{m}(x)
\begin{pmatrix}
1 & & \\
 & h & \\
 &   & 1
\end{pmatrix}gk\right)=\Lambda_F({}^t\lambda(-Q_{m-1})x)W(g)\\
\forall(x,h,g,k)\in F^{2m-2+n_0}\times H_{\lambda}\times G_{m}\times K_{m}
\end{matrix}\right.\right\},
\]
where $\Lambda_F$ denotes the additive character of $F$ trivial on ${\mathfrak o}$ but non-trivial on ${\mathfrak p}^{-1}$, where ${\mathfrak p}:=\varpi{\mathfrak o}$.

For $W\in\cW_{\lambda}^{\cF}$, $l\in\Z$ and a non-negative integer $k$ we put
\[
W_{k,l}:=W\left(
\begin{pmatrix}
\varpi^{k+l} & & \\
 & M_k & \\
 &        & \varpi^{-(k+l)}
\end{pmatrix}\right),
\]
{for which note that $W_{k,l}$ is checked to be zero for a negative $l$}.  
We see that $W$ is determined by the $W_{k,l}$'s in view of {the Iwasawa decomposition of $G_m$} and the following coset decomposition of $G_{m-1}$~(cf.~\cite[Lemma 7.2]{Su}): 
\begin{lem}
We have
\[
G_{m-1}=H_{\lambda}K_{m-1}\sqcup\underset{k\geq 1}{\bigsqcup}H_{\lambda}M_kK_{m-1}^*,
\]
where
\[
K_{m-1}^*:=\{h\in K_{m-1}\mid (h-1)L_{m-1}^{\sharp}\subset L_{m-1}\}.
\]
\end{lem}
We say that $W\in\cW_{\lambda}^{\cF}$ satisfies the local Maass relation if 
\begin{equation}\label{local-Maass-rel}
W_{k,l}-W_{k+1,l-1}=q^{-l}W_{k,0}\quad\forall k\ge 0,~\forall l\ge 0,
\end{equation}
which is equivalent to
\begin{equation}\label{local-Maass-rel-2}
W_{k,l}=\sum_{i=0}^{l}q^{-i}W_{k+l-i,0}.
\end{equation}
By $\cW_{\lambda}^{\cM}$ we denote the subspace of $\cW_{\lambda}^{\cF}$ consisting of those satisfying the local Maass relation.

We now review an explicit structure of $\cW_{\lambda}^{\cF}$ and $\cW_{\lambda}^{\cM}$ as $\cH_m$-modules. 
For that purpose let $L'_m:=\{x\in L_m^{\sharp}\mid \frac{1}{2}{}^tx Q_mx\in{\mathfrak p}^{-1}\}$ and denote by $\partial$ the dimension of $L'_m/L_m$ as a vector space over the residual field ${\mathfrak o}/{\mathfrak p}$ of $F$~(cf.~\cite[p6]{Su}). We furthermore put $q:=\sharp({\mathfrak o}/{\mathfrak p})$. For a non-negative integer $m$ we put
\begin{equation}\label{important-notation}
f_{m,j}:=\frac{q^{j-1}(q^{m-j+1}-1)(q^{m-j+n_0}+q^{\partial})}{q^j-1}\quad(\forall j\in\Z\setminus\{0\}).
\end{equation}
We note that this is a modification of what has been introduced at \cite[(7.11)]{Su}~(for this see also Remark \ref{Rem-Suganoformula} (2)). For a positive integers $m, r$, set $R_m^{(r)}
:=K_m/(K_m \cap c_m^{(r)}K_m(c_m^{(r)})^{-1})$, and let $|R_m^{(r)}|$ denote the cardinality of $R_m^{(r)}$. 
We have
\begin{equation}\label{important-notation-2}
|R_m^{(r)}|:=
\begin{cases}
\prod_{j=1}^rf_{m,j}&(1\le r\le m)\\
1&(r=0).
\end{cases}
\end{equation}

Without difficulty it is verified that $\cW_{\lambda}^{\cF}$ is stable under the action of $\cH_{m}$. In \cite[Corollary 7.5,~Corollary 7.8]{Su} the aforementioned explicit $\cH_{m}$-module structure of  $\cW_{\lambda}^{\cF}$ and $\cW_{\lambda}^{\cM}$ is given~(for this see Remark \ref{Rem-Suganoformula} (2)). We state it with the notation $f_{m,j}$ as follows:
\begin{prop}\label{Local-Whitt-Maass-sp}
\begin{enumerate}
\item Suppose that $m\ge 3$. On $\cW_{\lambda}^{\cF}$ the Hecke operators $C_{m}^{(r)}$ for $r\ge 3$ act as
\[
C_{m}^{(r)}=|R_{m-2}^{(r-2)}|\left(C_{m}^{(2)}-\frac{q^{r-2}-1}{q^{r-1}-1}\cdot f_{m-1,1}\cdot C_{m}^{(1)}+\frac{q^{r-2}-1}{q(q^r-1)}f_{m-1,1}f_{m+1,2}\right).
\]
\item The subspace $\cW_{\lambda}^{\cM}$ of $\cW_{\lambda}^{\cF}$ is stable under the action by $\cH_{m}$. In addition to the above formula, for $m\ge 2$, $C_{m}^{(2)}$ coincides with
\[
\begin{cases}
f_{m-1,1}C_{m}^{(1)}+q^4f_{m-2,1}f_{m-2,2}{-}q^3f_{m-2,1}^2-q^2(q^{2m-4+{n_0}}-(q-2)q^{\partial})f_{m-2,1}\\
+(q-1)q^{\partial}f_{m-1,1}-q(q^{2m-4+n_0}
+q^{\partial})^2&(m\ge 3),\\
f_{1,1}(C_2^{(1)}-\displaystyle\frac{q-1}{q^2-1}f_{2,1})&(m=2).
\end{cases}
\]
as Hecke operators acting on $\cW_{\lambda}^{\cM}$. 
\end{enumerate}
In particular, the Hecke eigenvalues of an eigenvector in $\cW_{\lambda}^{\cM}$ with respect to the $\cH_{m}$-action is determined by the eigenvalue of $C_{m}^{(1)}$.
\end{prop}
This is checked by a direct calculation. Regarding the formula in {the second assertion} for the case of $m=2$ we should note that $f_{0,1}$ is defined to be $0$, from which we deduce that the formula for $m=2$ is compatible with that for $m\ge 3$.

We are now able to point out that the description of the $\cH_m$-module structure of $\cW_{\lambda}^{\cM}$ above can be simplified  as follows:
\begin{prop}\label{Simple-Heckemodule}
Suppose that $m\ge 2$. 
As Hecke operators on $\cW_{\lambda}^{\cM}$, we have the coincidence 
\[
C_m^{(r)}=|R_{m-1}^{(r-1)}|\left(C_m^{(1)}-\frac{q^{r-1}-1}{q^r-1}f_{m,1}\right)
\]
for $2\le r\le m$.
\end{prop}
\begin{proof}
The case of $m=2$ is already shown in Proposition \ref{Local-Whitt-Maass-sp}, 2, for which note that $|R_1^{(1)}|=f_{1,1}$. 
The proof for the case of $m\ge 3$ starts with the following lemma:
\begin{lem}
If we assume that $C_m^{(2)}=|R_{m-1}^{(1)}|\left(C_m^{(1)}-\displaystyle\frac{q-1}{q^2-1}f_{m,1}\right)$ holds, we then have the formulas for $C_m^{(r)}$ with $r\ge 3$. 
\end{lem}
\begin{proof}
Insert the assumed formula for $C_m^{(2)}$ into those in Proposition \ref{Local-Whitt-Maass-sp}, 1. Furthermore note that 
\[
|R_{m-1}^{(r-1)}|=f_{m-1,1}\times p^{r-2}\cdot\frac{q-1}{q^{r-1}-1}\times|R_{m-2}^{(r-2)}|
\]
for $r\ge 3$. Then the lemma is settled by a direct calculation.
\end{proof}

What is remaining now is to deduce the formula for $C_m^{(2)}$ from that in Proposition \ref{Local-Whitt-Maass-sp}, 2. This needs the following technical lemma:
\begin{lem}
We have
\begin{align*}
& q^2f_{m-1,1}f_{m-1,2}-qf_{m-1,1}^2 \\
& \qquad =q^4f_{m-2,1}f_{m-2,2}-q^3f_{m-2,1}^2-q^2(q^{2m-4+n_0}-(q-2)q^{\partial})f_{m-2,1}-q(q^{2m-4+n_0}+q^{\partial})^2.
\end{align*}
\end{lem}
\begin{proof}
To show this we need
\begin{align*}
f_{m-1,1}&=qf_{m-2,1}+(q^{2m-4+n_0}+q^{\partial}),\\
f_{m-1,2}&=qf_{m-2,2}+\frac{q}{q+1}(q^{2m-6+n_0}+q^{\partial})=\frac{1}{q+1}(f_{m-1,1}-(p^{2m-4+n_0}+q^{\partial})).
\end{align*}
This verifies the coincidence of both sides.
\end{proof}

We therefore see that the formula in Proposition \ref{Local-Whitt-Maass-sp}, 2 implies 
\[
C_m^{(2)}=|R_{m-1}^{(1)}|(C_m^{(1)}+(q-1)+q^2f_{m-1,2}-qf_{m-1,1}),
\] 
for which note that $f_{m-1,1}=|R_{m-1}^{(1)}|$. 
This is verified to coincide with the desired formula for $C_m^{(2)}$ by a direct computation. 
As a result we have completed the proof of Proposition \ref{Simple-Heckemodule}.
\end{proof}

For the application to the action of the Hecke operators on $F_f$  we assume that $m=4n+1$ for $n\ge 1$, $F=\Q_p$, $q=p$ and $\partial=n_0=0$. 
We describe the actions of the Hecke operators $\{C_{4n+1}^{(r)}\mid 1\le r\le 4n+1\}$ on $W\in\cW_{\lambda}^{\cF}$ in terms of a recurrence relation of the $W_{k,l}$'s~(see \cite[Theorem 7.4]{Su}).
\begin {prop}\label{Hecke-recurrence}
Let $C_{4n+1}^{(r)}*W\in\cW_{\lambda}^{\cF}$ denote the action of $C_{4n+1}^{(r)}$ on $W\in\cW_{\lambda}^{\cF}$. For two non-negative integer $k,l$ we have the following formula:
\begin{align*}
(C_{4n+1}^{(r)}*W)_{k,l}=&|R_{4n-1}^{(r-2)}|\{p^{8n}(W_{k-1,l+2}+u_{r-1}W_{k,l+1}+p^{8n-2}W_{k+1,l})\\
&+(p(p^{r-2}-1)u_{r-2}+p^rf_{4n-1,r-1}u_r)W_{k,l}\\
&+pu_{r-1}W_{k-1,l+1}+p^{8n-1}u_{r-1}W_{k+1,l-1}\\
&+W_{k-1,l}+u_{r-1}W_{k,l-1}+p^{8n-2}W_{k+1,l-2}-\delta(l=0)p^{8n-1}W_{k,0}\\
&+\delta(k=0)\{p^{4n-1}\beta_{\lambda}(p^{8n}(W_{0,l+1}-W_{1,l})+pu_{r-1}(W_{0,l}-W_{1,l-1})\\
&+(W_{0,l-1}-W_{1,l-2}))+p^{8n}W_{0,l+1}+pu_{r-1}W_{0,l}+W_{0,l-1}\}\\
&+\delta(k=l=0)p^{4n}\beta_{\lambda}W_{0,0}\},
\end{align*}
where we put $u_r:=p^rf_{4n-1,r}+p^{r-1}-1$ and
\[
\beta_{\lambda}:=
\begin{cases}
0&(\text{$p$-adic order of $\frac{1}{2}{}^t\lambda J_{4n}\lambda=1$})\\
-1&(\text{$p$-adic order of $\frac{1}{2}{}^t\lambda J_{4n}\lambda=0$})
\end{cases}.
\] 
Here the formula for $r=1$ needs the following interpretation
\[
|R_{4n-1}^{(-1)}|=0,~|R_{4n-1}^{(-1)}|f_{4n-1,0}=|R_{4n-1}^{(-1)}|u_0=1.
\] 
We furthermore have $W_{k',l'}=0$ for negative $k',l'$.
\end{prop}
\begin{rem}\label{Rem-Suganoformula}
For this proposition we have two remarks on Sugano's formula \cite[Theorem 7.4]{Su} in the general case etc.\\
(1)~The formula need the notation ``$\rho_{\lambda}$'' as well as ``$\beta_{\lambda}$''~(for their definitions see \cite[Proposition 7.3,~(2.17)]{Su}). For {the proposition} we see $\rho_{\lambda}=0$. 
We further remark that the proof of Proposition \ref{Local-Whitt-Maass-sp} is based on Sugano's formula in the general case.\\
(2)~The formula for the case of ``$r=\nu+2$'' need the interpretation $f_{\nu,\nu+1}=0$, which is not referred to in \cite{Su}. This is a reason for our modified definition of $f_{m,j}$~(see also the remark just after Proposition \ref{Local-Whitt-Maass-sp}). In addition, we remark that ``$3\le r\le\nu+1$'' should be replaced by ``$3\le r\le\nu+2$'' in \cite[Corollary 7.5]{Su}. 
\end{rem}
\subsection{Hecke theory for our lift $F_f$}
We are now in a position to state our result on the Hecke theory for $F_f$. 
\begin{thm}\label{MainThm-first}
Suppose that $f$ is a Hecke eigenform and let $\lambda_p$ be the Hecke eigenvalue of $f$ at $p<\infty$.\\
(1)~$F_f$ is a Hecke eigenform.\\
(2)~Let $\mu_i$ be the Hecke eigenvalue with respect to the Hecke operator $C_{4n+1}^{(i)}$ for $1\le i\le 4n+1$. We have
\begin{align*}
\mu_1&=p^{4n}(\lambda_p^2-2)+pf_{4n,1}=p^{4n}(\lambda_p^2+p^{4n-1}+\cdots+p+p^{-1}+\cdots+p^{-(4n-1)}),\\
\mu_i=&|R_{4n}^{(i-1)}|\left(\mu_1-\frac{p^{i-1}-1}{p^i-1}f_{4n+1,1}\right),~(2\le i\le 4n+1).
\end{align*}
\end{thm}
\begin{proof}
To apply Sugano's local theory as in Section \ref{Sugano-local} we need the following:
\begin{lem}\label{local-sp-Bessel}
Let us fix a prime $p$. 
Suppose that $\lambda\in\Q^{8n}$ is {reduced as an element in the maximal lattice} {$(\Z_p^{8n},J_{4n})$~(or $(\Z_p^{8n},-S)$). 
For this we remark that $(\Z_p^{8n},-S)$ is verified to be $GL_{8n}(\Z_p)$-equivalent to}
{$(\Z_p^{8n},J_{4n})$ by a standard argument using the theory of quadratic forms over p-adic  fields.}
\begin{enumerate}
\item As a function on $G_p$, $A_{\lambda}(g)\in\cW_{\lambda}^{\cM}$ for $g\in G_p$, where we regard $g$ as an element in $\cG(\A_f)$ in the usual manner.\\
\item For non-negative integers $l,m$ and a Hecke operator $C\in\cH_{4n+1}$ we have
\[
(C*A_{\lambda})\left(
\begin{pmatrix}
p^{l+m} & & \\
 & M_{m} & \\
 &          & p^{-(l+m)}
\end{pmatrix}\right)=p^{-4n(l+m)}(C*A_{p^{l+m}M_m^{-1}\lambda})(1_{8n+2}).
\]
\end{enumerate}
\end{lem}
\begin{proof}
We can check that $A_\lambda$ satisfies the local Maass relations (\ref{local-Maass-rel-2}) directly.
We only have to prove $A_{\lambda}(g)\in\cW^{\cF}_{\lambda}$ for $g\in G_p$. It suffices to prove
\[
A_{\lambda}\left(
\begin{pmatrix}
1 & & \\
 & h_0h & \\
 &      & 1 
\end{pmatrix}\right)=A_{\lambda}\left(
\begin{pmatrix}
1 & & \\
  & h & \\
  &   & 1
\end{pmatrix}\right)\quad\forall(h_0,h)\in H_{\lambda}\times\cH(\Q_p). 
\]
To this end write $h_0h$ as $h_0h=a_0u_0^{-1}$ with $a_0\in GL_{8n}(\Q)$ and $u_0\in\prod_{p<\infty}SL_{8n}(\Z_p)\times SL_{8n}(\R)$. Then, for every prime $p$, {the condition $a_0^{-1}\lambda=u_0^{-1}(h_0h)^{-1}\lambda\in\Z_p^{8n}$} implies the greatest power of $p$ dividing the entries in $a_0^{-1}\lambda$ is the same as that of $(h_0h)^{-1}\lambda=h^{-1}\lambda$~(which equals to ``$a^{-1}\lambda\in\Z^{8n}$'' in the notation of Section \ref{Adelic-lift}). We thereby see that the greatest common divisor $d_{\lambda}$ for $L_{h_0h}$ coincides that for $L_h$, and this proves part 1 of the lemma.

In view of the right $K_p$-invariance of $A_{\lambda}$ and the Iwasawa decomposition of $G_p$,  {the} part 2 of the lemma is reduced to showing
\[
A_{\lambda}\left(
\begin{pmatrix}
p^{l+m} & & \\
 & M_m & \\
 &        & p^{-(l+m)}
\end{pmatrix}n(x)
\begin{pmatrix}
1 & & \\
  & h & \\
  &    & 1
\end{pmatrix}\right)=p^{-4n(l+m)}A_{p^{l+m}M_m^{-1}\lambda}\left(n(x)
\begin{pmatrix}
1 & & \\
  & h & \\
  &    & 1
\end{pmatrix}\right)
\]
for any $(x,h)\in\Q_p^{2n}\times\cH(\Q_p)$. 
Reviewing the definition of $A_{\lambda}$, this is verified by a direct calculation.
\end{proof}

To prove the first assertion of the theorem it suffices to prove that $F_f$ is a Hecke eigenform with respect to the Hecke operator $C_{4n+1}^{(1)}$ by virtue of Proposition \ref{Local-Whitt-Maass-sp}{ (or Proposition \ref{Simple-Heckemodule})} and {part 1 of} Lemma \ref{local-sp-Bessel}.
\begin{prop}\label{Hecke-eigen-1}
For  $\lambda\in\Q^{8n}\setminus\{0\}$ we have
\[
(C_{4n+1}^{(1)}*A_{\lambda})(1_{8n+2})=p^{4n}(\lambda_p^2+p^{4n-1}+\cdots+p+p^{-1}+\cdots+p^{-(4n-1)})A_{\lambda}(1_{8n+2}).
\]
\end{prop}
\begin{proof}
We fix an arbitrary prime $p$ and may assume that $\lambda\in\Z_p^{8n}$. 
For the proof of the proposition, the following lemma is crucial.
\begin{lem}\label{Main-lemma-prop}
(1)~The Fourier coefficients $c(n)$ of a Hecke-eigen cusp form $f$ satisfy the following relations:
\begin{align*}
pc(p^2n)&=(\lambda_p^2-1)c(n)-{
\begin{cases}
p^{-\frac{1}{2}}\lambda_pc(n/p)&(p|n)\\
0&(p\nmid n)
\end{cases}},\\
pc(p^2n)&=(\lambda_p^2-2)c(n)-p^{-1}c(n/p^2),
\end{align*}
where we assume {$p^2|n$} for the second formula.\\
(2)~For a reduced $\lambda$ we have $
(C_{4n+1}^{(1)}*A_{\lambda})(
\begin{pmatrix}
p^{l+m} & & \\
 & M_m & \\
 &  & p^{-(l+m)}
\end{pmatrix})$
\begin{align*}
=&p^{8n}\cdot p^{-4n(l+m+1)}A_{p^{l+1}(p^mM_m^{-1}\lambda)}(1_{8n+2})+p^2f_{4n-1,1}\cdot p^{-4n(l+m)}A_{p^l(p^mM_m^{-1}\lambda)}(1_{8n+2})\\
&+p\cdot p^{-4n(l+m)}A_{p^{l+1}(p^{m-1}M_{m-1}^{-1}\lambda)}(1_{8n+2})+p^{8n-1}\cdot p^{-4n(l+m)}A_{p^{l-1}(p^{m+1}M_{m+1}^{-1}\lambda)}(1_{8n+2})\\
&+p^{-4n(l+m-1)}A_{p^{l-1}(p^mM_m^{-1}\lambda)}(1_{8n+2})\\
&+\delta(m=0)\{p^{4n-1} \beta_\lambda p\cdot (p^{-4nl}A_{p^l\lambda}(1_{8n+2})-p^{-4nl}A_{p^{l-1}(pM_1^{-1}\lambda)}(1_{8n+2}))\\
&+p\cdot p^{-4nl}A_{p^l\lambda}(1_{8n+2})\}.
\end{align*}
\end{lem}

The first assertion is a consequence of the well-known Hecke theory for modular forms of one variable~(cf.~\cite[Section 8.5]{Iwn}). Taking {part 2 of Lemma \ref{local-sp-Bessel} into account}, we see that the second assertion is nothing but a reformulation of Proposition \ref{Hecke-recurrence} for the case of $r=1$. 

In view of Lemma \ref{prim-gcd-red} and {part 2 of Lemma \ref{local-sp-Bessel}} we know that Lemma \ref{Main-lemma-prop} (2) describes the action of $C_{4n+1}^{(1)}$ on $A_{\lambda}$ for any $\lambda\in\Q^{8n}\setminus\{0\}$~(or any  $\lambda\in\Z_p^{8n}\setminus\{0\}$). 
As a result we verify the proposition by using Lemma \ref{Main-lemma-prop} (1) and the explicit expression of $A_{\lambda}$ in terms of the Fourier coefficients $c(n)$s. 
\end{proof}

To complete the proof of the theorem, we are left with proving the formula for the other Hecke eigenvalues $\mu_i$~($i\ge 2$). 
This is an immediate consequence from Proposition \ref{Simple-Heckemodule} and {part 1 of Lemma} \ref{local-sp-Bessel}.
As a result we are done.
\end{proof}

\section{Cuspidal representations generated by the lifts}\label{Cusp-rep}
\subsection{The archimedean component of a cuspidal representation generated by a Maass cusp form on $\cG(\A)$}
We let $\Gamma$ be an arithmetic subgroup of $O(Q)$ and suppose that $F\in M(\Gamma,r)$ can be adelized to be a cusp form of the adele group $O(Q)(\A)$, for which we do not have to assume for a moment that $(\Z^N,S)$ is even unimodular. 
By $\pi_F$ we denote the cuspidal representation generated by $F$.
\begin{prop}\label{Irreducibility-cusprep}
If $F$ is Hecke eigenvector at every finite place, $\pi_F$ is irreducible.
\end{prop}
\begin{proof}
We use \cite[Theorem 3.1]{N-P-S}, which reduces the problem to the irreducibility of the archimedean local representation of $\pi_F$.  
We first note that, as is well-known, each irreducible cuspidal representation occurs with finite multiplicity in the cuspidal spectrum, which implies that $\pi_F$ is a finite sum of irreducible cuspidal representations. We therefore see that its archimedean representation is also a finite sum of irreducible admissible representations of  $G_{\infty}$~($=O(Q)(\R)$). Let us now note that $F$ is right $K_{\infty}$-invariant, which means that $F$ generates the trivial representation as a $K_{\infty}$-module. The trivial $K_{\infty}$-module should occur with at most multiplicity one in an irreducible admissible representation of $G_{\infty}$. In fact, in view of the Langlands classification \cite{La}, an irreducible admissible representation with the trivial $K_{\infty}$-representation can be embedded into a spherical principal series representation, which has the trivial $K_{\infty}$-representation with multiplicity one. 
We thereby know that $F$ should be inside only one irreducible admissible representation of $G_{\infty}$.     
\end{proof}

We next determine the irreducible admissible representation of $G_{\infty}$ which $F$ generates. To this end, let $\delta_s:A_{\infty}\rightarrow\C^{\times}$ be the quasi-character parametrized by $s\in\C$ given by the formula $\delta_s(y) = y^s$. {We trivially extend $\delta_s$ to a quasi character of the standard proper parabolic subgroup $P_{\infty}$},
{which admits a Langlands decomposition $N_{\infty}A_{\infty}M_{\infty}$ with Levi subgroup}
{$\left\{\left.\left(\begin{smallmatrix}
1 & {}^t0_N & 0\\
0_N & m & 0_N\\
0 & {}^t0_N & 1
\end{smallmatrix}\right)~\right|~m\in U_{\infty}(=O(S)(\R))\right\}$}. By $I_{P_{\infty}}^{G_{\infty}}(\delta_s)$ we  denote the normalized parabolic induction defined by $\delta_s$.
We remark that every spherical principal series representation of $G_{\infty}$ is of this form. We now have the following:
\begin{prop}\label{archimedean-rep}
Suppose that $F\in M(\Gamma,\sqrt{-1}r)$ with $r\in\R$. The archimedean component of $\pi_F$ is isomorphic to $I_{P_{\infty}}^{G_{\infty}}(\delta_{\sqrt{-1}r})$ as admissible $G_{\infty}$-modules, and is tempered.
\end{prop}
\begin{proof}
Let $v$ be the $K_{\infty}$-vector of $I_{P_{\infty}}^{G_{\infty}}(\delta_s)$, which is unique up to scalars. If the archimedean representation is embedded into this parabolic induction, the vector $v$ has the same eigenvalue for the Casimir operator $\Omega$ as $F$. We then verify that
\[
\Omega\cdot v=\frac{1}{2N}(s^2-\frac{N^2}{4})v=-\frac{1}{2n}(r^2+\frac{N^2}{4})v,
\]
and therefore $s=\pm\sqrt{-1}r$. By the same reasoning as in \cite[Proposition 6.5,~Theorem 6.8]{Mu-N-P} we then see that the proposition is a consequence of \cite[Section 41,~Theorem 1]{Ha} and \cite[Remark (2.1.13)]{Co}.
\end{proof}
\subsection{Unramified principal series representations at non-archimedean places}
For unramified characters $\chi_i$ of $\Q_p^{\times}, 1\le i\le 4n+1$, denote an unramified character of the standard split torus~($\simeq(\Q_p^{\times})^{4n+1}$) by $\chi=\diag(\chi_1,\chi_2,\cdots,\chi_{4n+1},\chi_{4n+1}^{-1},\cdots,\chi_2^{-1},\chi_1^{-1})$. Let $I(\chi)$ be the normalized parabolic induction of $G_p\simeq G_{4n+1}$ induced from the character of the Borel subgroup defined by $\chi$. The representation of $G_p$ given by $I(\chi)$ is called an unramified principal series representation. 
For us it is important to review the fundamental properties of unramified principal series representations of $G_p\simeq G_{4n+1}$. In many references $p$-adic reductive groups are often assumed to be connected for the convenience of the argument to study such representations. 
Though $G_p$ is not connected, we can say that such fundamental properties are still valid for $G_p$. We need the following lemma:
\begin{lem}\label{unramified-rep}
\begin{enumerate}
\item For any unramified character $\chi$, the unramified principal series representation {$I(\chi)$} has a unique irreducible subquotient with a $K_p$-invariant vector~(called a spherical vector). 
Conversely, any irreducible admissible representation of $G_p$ with a spherical vector~(called an irreducible unramified representation) is given by the irreducible subquotient of an unramified principal series representation.
\item Two irreducible unramified representations are isomorphic to each other if and only if the Hecke eigenvalues of the spherical vectors of the two representations are the same.
\end{enumerate}
\end{lem}
\begin{proof} 
We first note that, as an admissible $G_p$-module, every irreducible unramified representation is isomorphic to the space of functions on $G_p$ generated by a zonal spherical function and that every zonal spherical function is uniquely parametrized by an unramified character of the split torus of $G_p$ modulo the conjugation by the Weyl group~(cf.~\cite[Theorem 4.3]{Ca},~\cite[Theorem 2]{Sa}). This implies the second  assertion of part 1 of the lemma since the $G_p$-module generated by a zonal spherical function can be embedded into an unramified principal series representation with the unramified character parametrizing the zonal spherical function. 
Part 2 of the lemma follows from the bijective correspondence between unitary algebra  homomorphisms of the Hecke algebra $\cH_{4n+1}$ to $\C$ and the equivalence classes of unramified characters of the split torus by the Weyl group conjugation~(cf.~\cite[Corollary 4.2]{Ca}), the latter of which parametrize the equivalence classes of the irreducible unramified representations.  We should note that the disconnected-ness of $G_p$ has no influence for these consequences~(see \cite[Theorem 5,~Remark 1 after Theorem 3]{Sa} and \cite[Section 4,4]{Ca}). 
In fact, Satake's theory on the Hecke algebras and the zonal spherical functions hold also for $G_p$, which has the commutative Hecke algebra $\cH_{4n+1}$ and admits the Iwasawa and Cartan decompositions.

We are now left with the first assertion of part 1 of the lemma. The Frobenius reciprocity of induced representations implies that the unramified principal series representation restricted to $K_p$ has the trivial representation of $K_p$ with multiplicity one, which leads to the first assertion of part 1 of the lemma, since there is a contradiction to the multiplicity one condition unless the uniqueness of the irreducible subquotient with a spherical vector holds. In fact, every {irreducible} unramified representation has a unique spherical vector, up to constant multiples, since it admits a unique zonal spherical function.  
\end{proof}

We call the unique irreducible subquotient of an unramified principal series representation the {\it spherical constituent}.
\subsection{Cuspidal representation generated by $F_f$}
We are now able to show the result on the explicit determination of the cuspidal representation generated by $F_f$ as follows:
\begin{thm}\label{MainThm-second}
Let $f$ be a Hecke eigenform and let $\pi_{F_f}$ be the cuspidal representation generated by $F_f$.
\begin{enumerate}
\item The representation $\pi_{F_f}$ is irreducible and thus has the decomposition into the restricted tensor product $\otimes'_{v\le\infty}\pi_v$ of irreducible admissible representations $\pi_v$.
\item For $v=p<\infty$, $\pi_p$ is the spherical constituent of the unramified principal series representation of $G_p$ with the Satake parameter
\[
\diag \left(\left(\frac{\lambda_p+\sqrt{\lambda_p^2-4}}{2}\right)^2,p^{4n-1},\cdots,p,1,1,p^{-1},\cdots,p^{-(4n-1)},\left(\frac{\lambda_p+\sqrt{\lambda_p^2-4}}{2}\right)^{-2}\right).
\]
\item For every finite prime $p<\infty$, $\pi_p$ is non-tempered while $\pi_{\infty}$ is tempered.
\end{enumerate}
\end{thm}
\begin{proof}
This first assertion is a consequence of Theorem \ref{MainThm-first}  and Proposition \ref{Irreducibility-cusprep}. To prove the other two assertions we fix an arbitrary prime $p$. Since $F_f$ is right $K_p$-invariant for each finite prime $p$, $\pi_p$ has to be the spherical constituent of an unramified principal series representation~(cf.~Lemma \ref{unramified-rep} (1)). 

The Hecke eigenvalue $\mu_1$ of $F_f$ enables us to {suppose} that  
\begin{align*}
&\diag(\chi_1(p),\chi_2(p),\cdots,\chi_{4n+1}(p),\chi_{4n+1}(p)^{-1},\cdots,\chi_2(p)^{-1},\chi_1(p)^{-1})=\\
&\diag\left(\left(\frac{\lambda_p+\sqrt{\lambda_p^2-4}}{2}\right)^2,p^{4n-1},\cdots,p,1,1,p^{-1},\cdots,p^{-(4n-1)},\left(\frac{\lambda_p+\sqrt{\lambda_p^2-4}}{2}\right)^{-2}\right)
\end{align*}
as the second assertion indicates. 
Since irreducible unramified representations are determined by Hecke eigenvalues of the spherical vectors up to equivalence (cf.~Lemma \ref{unramified-rep} 2.) we need to show the Hecke eigenvalue of a spherical vector of $\pi_p$ for $C_{4n+1}^{(i)}$ coincides with $\mu_i$ for each $i\ge 1$. 
This is nothing but the second assertion. The last assertion for $\pi_p$ follows from the second one since the Satake parameter tells us that {the representation $(\pi_p|_{G_{i}},\C v)$ with a spherical vector $v$} is the trivial representation of $G_i$ for $i\le 4n$. 
In fact, to show the non-temperedness of $\pi_p$, consider the $(2+\epsilon)$-th power of the matrix coefficient of the spherical vector for each $\epsilon>0$. For any fixed positive integer $i\leq 4n$ the integral of this over $\{K_pgK_p\mid g\in G_i\}$ is then divergent. 
For this we remark that $G_i$ is viewed as a subgroup of $G_p\simeq G_{4n+1}$ by embedding it naturally into the middle {$2i\times 2i$}-block of $G_p$.
The tempered-ness of $\pi_{\infty}$ is nothing but Proposition \ref{archimedean-rep}. 

The rest of the proof is thus devoted to studying the Hecke eigenvalues. 
For that purpose we state the following:
\begin{lem}
For $1\le m\le 4n+1, 1\le i\le m$ let $\phi_m(C_{m}^{(i)})$ denote the $C_{m}^{(i)}$-action on the spherical vector of $\pi_p|_{G_m}$.  
For $i\ge 2$ we have
\[
\phi_{4n+1}(C_{4n+1}^{(i)})=p^{4n}(\chi_1(p)+\chi(p)^{-1})\phi_{4n}(C_{4n}^{(i-1)})+(p^{i-1}-1)f_{4n,i-1}\phi_{4n}(C_{4n}^{(i-2)})+p^i\phi_{4n}(C_{4n}^{(i)}).
\]
\end{lem}
\begin{proof}
Noting the parabolic induction which contains $\pi_p$ as the spherical constituent, we can deduce this from the coset decomposition of $K_pc_{4n+1}^{(i)}K_p$ in \cite[Lemma 7.1]{Su}.
\end{proof}

Reviewing the normalization of the invariant measure as in the beginning of Section \ref{Hecke-theory}, we see
\[
\phi_{4n}(C_{4n}^{(i)})=|R_{4n}^{(i)}|~(1\le i\le 4n)
\]
since the action $\pi_p|_{G_m}$ on the spherical vector is trivial for $m\le 4n$ as is explained above. In addition, we also note
\[
\frac{|R_{4n}^{(i)}|}{|R_{4n}^{(i-1)}|}=f_{4n,i}~(1\le i\le 4n).
\]
Hence, from this lemma we get
\[
\phi_{4n+1}(C_{4n+1}^{(i)})=
\begin{cases}
\mu_1&(i=1),\\
|R_{4n}^{(i-1)}|(\mu_1+(p^{i-1}-1)+p^if_{4n,i}-pf_{4n,1})&(i\ge 2).
\end{cases}
\]
Without difficulty we then verify by a direct calculation that the value of $\phi_{4n+1}(C_{4n+1}^{(i)})$ coincides with $\mu_i$. Consequently we have proved Theorem \ref{MainThm-second}.
\end{proof}

As a result of this theorem~(or as a result of Theorem \ref{MainThm-first} and \cite[Corollary 7.9]{Su}) we can write down the standard $L$-function of $\pi_{F_f}$~(or $F_f$).
\begin{cor}\label{Std-L-fct}
For any prime $p$ the local $p$-factor $L_p(\pi_{F_f},\operatorname{St},s)$ of the standard $L$-function for $\pi_{F_f}$~(or $F_f$) is written as
\begin{align*}
L_p(\pi_{F_f},\operatorname{St},s)&=\zeta_p(s)(1-(\lambda_p^2-2)p^{-s}+p^{-2s})^{-1}\prod_{j=0}^{8n-2}\zeta_p(s+j-(4n-1))\\
&=L_p(\sym^2(f),s)\prod_{j=0}^{8n-2}\zeta_p(s+j-(4n-1)),
\end{align*}
where $\zeta_p$ denotes the $p$-factor of the Riemann zeta function and $L_p(\sym^2(f),s)$ is the $p$-factor of the symmetric square $L$-function for $f$.
\end{cor}
\section{Appendix: Cuspidal representations generated by Oda-Rallis-Sciffmann lifts}
We have used Sugano's non-archimedean local theory in \cite[Section 7]{Su} to study {the} Hecke theory of the cusp forms given by our lifting and the cuspidal representations generated by them. 
His local theory is originally motivated by studying non-archimedean local aspect of the lifting theory of Oda \cite{Od} and Rallis-Schiffman \cite{Ra-Sch}. We can therefore expect that the results in Section \ref{Hecke-theory} and \ref{Cusp-rep} naturally hold also for the lifting by Oda and Rallis-Schiffman. In this appendix, still restricting ourselves to ``the case of even unimodular lattices'', we carry out the argument similar to Sections \ref{Hecke-theory} and \ref{Cusp-rep} to deduce similar results for the case of Oda-Rallis-Schiffman lifting. 
\subsection{Basic notation}
Let $(\Z^{8n},S)$ be an even unimodular lattice with a positive definite symmetric matrix $S$ and put $Q_1:=
\begin{pmatrix}
 & & 1\\
 & -S & \\
1 & &
\end{pmatrix}$. We then let $Q_2:=
\begin{pmatrix}
 & & 1\\
 & Q_1 & \\
1 & &
\end{pmatrix}$ and let $\cG=O(Q_2)$~(respectively $\cH=O(Q_1)$) be the orthogonal group {over $\Q$} defined by $Q_2$~(respectively~$Q_1$). We introduce several algebraic subgroups of $\cG$. We first introduce the maximal parabolic subgroup $\cP$ with a Levi decomposition $\cP=\cN_1\rtimes\cL_1$, where $\cN_1$ and $\cL_1$ are defined by the groups of $\Q$-rational points as follows:
\begin{align*}
\cN_{1}(\Q)&=\left\{\left.
{n_{Q_1}(x)}=\begin{pmatrix}
1 & {}^txQ_1 & \frac{1}{2}{}^txQ_1x\\
  & 1_{8n+2} & x\\
 &     & 1
\end{pmatrix}~\right|~x\in\Q^{8n+2}\right\},\\
\cL_{1}(\Q)&=\left\{\left.
\begin{pmatrix}
a & & \\
 & h & \\
 &  & a^{-1}
\end{pmatrix}~\right|~a\in\Q^{\times},~h\in O(Q_1)(\Q)\right\}.
\end{align*}
For $w\in\Q^{8n}$ let $n_0(w):=
\begin{pmatrix}
1 & {}^twS & \frac{1}{2}{}^twSw\\
 & 1_{8n} & w\\
 & & 1
\end{pmatrix}$ and $n_{1}(w):=
\begin{pmatrix}
1 &  & \\
 & n_0(w) & \\
 &   & 1
\end{pmatrix}$. We then introduce the maximal unipotent subgroup $\cN$ of $\cG$ defined by its group of $\Q$-rational points
\[
\cN(\Q):=\{n(x,w)\mid x\in\Q^{8n+2},~w\in\Q^{8n}\},
\]
where $n(x,w):=n_{Q_1}(x)n_{1}(w)$.

Let $G_{\infty}$ be the real Lie group $\cG(\R)$. To describe an Iwasawa decomposition of $G_{\infty}$ we introduce 
\[
A_{\infty}:=\left\{\left.
\begin{pmatrix}
a_1 & & & & \\
  & a_2 & & & \\
  & & 1_{8n} & & \\
 & & & a_2^{-1} & \\
 & & & & a_1^{-1}
\end{pmatrix}~\right|~a_1,~a_2\in\R^{\times}_+\right\}
\]
and a maximal compact subgroup 
\[
K_{\infty}:=\{k\in G_{\infty}\mid {}^tkRk=R\} 
\]
of $G_{\infty}$, where $R=
\begin{pmatrix}
1_2 & & \\
 & S & \\
&  & 1_2
\end{pmatrix}$ is the majorant of $Q_2$. We then have an Iwasawa decomposition of $G_{\infty}$ as follows:
\[
G_{\infty}=\cN(\R)A_{\infty}K_{\infty}.
\]

We next introduce the symmetric domain of type $IV$, which is identified with the quotient $G_{\infty}/K_{\infty}$. 
We follow \cite[Section 1.4]{Na} to describe it. 
  Let $B_{Q_1}$ be the bilinear form on $V\times V$ defined by $Q_1$ with $V=\R^{8n+2}$ and let $(V,\tau)$ be the Euclidean Jordan algebra equipped with the trace form
\[
\tau:V\times V\ni(x,y)\mapsto\tau(x,y)=2B_{Q_1}(x\circ y,e),
\]
where 
\[
x\circ y:=B_{Q_1}(x,e)y+B_{Q_1}(y,e)x-B_{Q_1}(x,y)e,\quad(x,y\in V)
\]
with ${}^te=(\frac{1}{\sqrt{2}},0,\cdots,0,\frac{1}{\sqrt{2}})$. This Euclidean Jordan algebra has the determinant $\Delta$ given by
\[
\Delta(v):=\frac{1}{2}B_{Q_1}(v,v)\quad(v\in V).
\]
Let us introduce the symmetric cone $\Omega:=\{v\in V\mid B_{Q_1}(v,v)>0,~B_{Q_1}(v,e)>0\}$ of $V$. Then the symmetric domain of type $IV$ corresponding to $G_{\infty}$ is realized as $\cD:=V+\sqrt{-1}\Omega$. 
The Lie group $G_{\infty}$ acts on $\cD$ by the linear fractional transformation, for which we use the notation $g\cdot z$ for $(g,z)\in G_{\infty}\times\cD$. Let $J(g,z)\in\C$ be the automorphy factor for $(g,z)\in G_{\infty}\times\cD$. For the definition of $g\cdot z$ and $J(g,z)$ see \cite[Section 1]{Gr}. 
We can identify $G_{\infty}/K_{\infty}$ with $\cD$ by
$G_{\infty}\ni g\mapsto g\cdot(\sqrt{-1}e)\in\cD$. 
\subsection{Review on Oda-Rallis-Schiffman lifting}
By $S_{\kappa}(SL_2(\Z))$ we denote the space of holomorphic cusp forms on ${\mathfrak h}$ of weight $\kappa$ with respect to $SL_2(\Z)$. 
To review the Oda-Rallis-Schiffman lift from these holomorphic cusp forms we introduce the archimedean Whittaker function $W_{\lambda,\kappa}$ on $G_{\infty}$ with $\lambda\in\Omega$ and a positive integer $\kappa$ by
\begin{align*}
& W_{\lambda,\kappa}(n(x,w)ak)\\
&\qquad :=J(k_{\infty},\sqrt{-1}e)^{-\kappa}\Delta({\rm Im}({n_{1}(w)}a\cdot\sqrt{-1}e))\exp(2\pi\sqrt{-1}\tau(\lambda,x+\sqrt{-1}{\rm Im}({n_{1}(w)}a\cdot\sqrt{-1}e))
\end{align*}
for $(x,w,a,k)\in\R^{8n+2}\times\R^{8n}\times A_{\infty}\times K_{\infty}$, where ${\rm Im}(z)$ denotes the imaginary part of $z\in\cD$. 

Let $f\in S_{\kappa-4n+2}(SL_2(\Z))$ be given by the $q$-expansion $f(\tau)=\sum_{m\geq 1}c(m)q^m$~(thus $\kappa$ has to be even and $\kappa-4n+2\ge 12$). We put $|\lambda|_{Q_1}:=\sqrt{\frac{1}{2}{}^t\lambda Q_1\lambda}=\sqrt{\Delta(\lambda)}$ for $\lambda\in V$.  We introduce a smooth function $F_f$ on $G_{\infty}$ by
\[
F_f(g_{\infty})=\sum_{\lambda\in\Z^{8n+2}\cap\Omega}C_{\lambda}W_{\lambda,\kappa}(g_{\infty}),
\]
where
\[
C_{\lambda}=\sum_{d|d_{\lambda}}d^{\kappa-1}c(\frac{|\lambda|_{Q_1}^2}{d^2})
\]
with the greatest common divisor $d_{\lambda}$ of the entries of $\lambda$. 
For the maximal lattice $\Z^{\oplus 2}\oplus\Z^{8n}\oplus\Z^{\oplus 2}$ with respect to $Q_2$ we introduce an arithmetic subgroup
\[
\Gamma_S:=\{\gamma\in\cG(\Q)\mid \gamma(\Z^{\oplus 2}\oplus\Z^{8n}\oplus\Z^{\oplus 2})=\Z^{\oplus 2}\oplus\Z^{8n}\oplus\Z^{\oplus 2}\}.
\]
We now state the following theorem by Oda \cite[Corollary to Theorem 5]{Od} and Rallis-Schiffman \cite[Theorem 5.1]{Ra-Sch}.
\begin{thm}[Oda, Rallis-Schiffman]
For $\kappa>8n+4$ the smooth function $F_f$ is a holomorphic cusp form of weight $\kappa$ with respect to $\Gamma_S$, lifted from the domain $\cD$ to the group $G_{\infty}$.
\end{thm}
\subsection{Cuspidal representations generated by Oda-Rallis-Schiffman lifts.}
\subsection*{(1)~Adelization of $F_f$.}
To consider the cuspidal representation generated by $F_f$ we adelize $F_f$. We carry out it following the argument similar to Section \ref{Adelic-lift}. 

Let $K_f:=\prod_{p<\infty}K_p$ with $K_p:=\{g\in\cG(\Q_p)\mid g\Z_p^{8n+4}=\Z_p^{8n+4}\}$. We remark that the strong approximation theorem of $\cG(\A)$ with respect to the maximal compact subgroup $K_f$ holds, from which we deduce that the set of $\Gamma_S$-cusps is in bijection with $\cH(\Q)\backslash\cH(\A)/\cH(\R)U_f$ with $U_f:=\prod_{p<\infty}U_p~(U_p:=\{h\in\cH(\Q_p)\mid h\Z_p^{8n+2}=\Z_p^{8n+2}\})$. 
This is nothing but Lemma \ref{Classnum-Cusps} for the case of $\cG=O(Q_2)$. 

For $h=(h_p)_{p\le\infty}\in\cH(\A_f)$ we put $L_h:=(\prod_{p<\infty}h_p\Z_p^{8n+2}\times\R^{8n+2})\cap\Q^{8n+2}$ and write $h=au^{-1}$ with $(a,u)\in GL_{8n+2}(\Q)\times(\prod_{p<\infty}SL_{8n+2}(\Z_p)\times SL_{8n+2}(\R))$. 
For $\lambda\in L_h\setminus\{0\}$ we denote by $d_{\lambda}$ the greatest common divisor of the entries of $a^{-1}\lambda\in\Z^{8n+2}$, which is checked to be well-defined by the same argument as the proof of Lemma \ref{Well-def-smfct}. 
 
We introduce a function $A_{\lambda}$ indexed by $\lambda\in\Q^{8n+2}\setminus\{0\}$ as follows:
\begin{align*}
A_{\lambda}\left(
\begin{pmatrix}
1 & & \\
 & h & \\
 &  & 1
\end{pmatrix}\right)&:=
\begin{cases}
\sum_{d|d_{\lambda}}d^{\kappa-1}c(\frac{|\lambda|_{Q_1}^2}{d^2})&(\lambda\in L_h)\\
0&(\lambda\in\Q^{8n+2}\setminus L_h)
\end{cases},\\
A_{\lambda}\left(
\begin{pmatrix}
\beta & & \\
 & h & \\
 &  & \beta^{-1}
\end{pmatrix}\right)&:=||\beta||_{\A}^{\kappa}A_{||\beta||_{\A}^{-1}\lambda}\left(
\begin{pmatrix}
1 & & \\
 & h & \\
 &  & 1
\end{pmatrix}\right)\quad\forall(\beta,h)\in\A_f^{\times}\times\cH(\A_f),\\
A_{\lambda}(n_{2}(x)gk)&:=\Lambda({}^t\lambda Q_1x)A_{\lambda}(g)\quad\forall(x,g,k)\in\A_f^{8n+2}\times\cG(\A_f)\times K_f,
\end{align*}
where $\Lambda$ denotes the standard additive character of $\A/\Q$. This $A_{\lambda}$ is verified to be well-defined function on $\cG(\A_f)$ similarly as in the proof of Lemma \ref{Well-def-smfct}.
With this $A_{\lambda}$ we adelize $F_f$ by
\[
F_f(g)=\sum_{\lambda\in\Q^{8n+2}\setminus\{0\}}A_{\lambda}(g_f)W_{\lambda,\kappa}(g_{\infty})
\]
for $g=g_fg_{\infty}\in\cG(\A)$ with $(g_f,g_{\infty})\in\cG(A_f)\times G_{\infty}$. 
By the definition of the adelized $F_f$, $F_f$ is right $K_f$-invariant. 
By the standard argument in terms of the strong approximation theorem the left $\Gamma_S$-invariance of the non-adelic $F_f$ then implies the left $\cG(\Q)$-invariance of the adelic $F_f$. The adelized $F_f$ is a cusp form on $\cG(\A)$.
\subsection*{(2)~Cuspidal representation generated by $F_f$}
To determine explicitly the cuspidal representation of $\cG(\A)$ generated by $F_f$ we first provide an explicit formula for Hecke eigenvalues of  the adelized $F_f$. We can apply the non-archimedean local theory in Section \ref{Sugano-local} to our situation that $m=4n+2,~F=\Q_p,~q=p$ and $\partial=n_0=0$. The $p$-adic group $\cG(\Q_p)$ is viewed as $G_{4n+2}$ in the notation of Section \ref{Sugano-local}. We need the lemma as follows:
\begin{lem}
As a function on $\cG(\Q_p)(\simeq G_{4n+2})$, $A_{\lambda}(g)\in{\cW}^{\cM}_{\lambda}$, where we regard $g\in\cG(\Q_p)$ as an element in $\cG(\A)$ in the usual manner.
\end{lem}
We then state the theorem on the Hecke eigenvalues of $F_f$.
\begin{thm}\label{Hecke-eigenval-Oda-RS}
Suppose that $f$ is a Hecke eigenform and let $\lambda_p$ be the Hecke eigenvalue of $f$ at $p<\infty$.\\
(1)~$F_f$ is a Hecke eigenform.\\
(2)~Let $\mu_i$ be the Hecke eigenvalue of the Hecke operator for $C_{4n+2}^{(i)}$ with $1\le i\le 4n+2$. We have
\[
\mu_i=
\begin{cases}
p^{4n+1}(p^{-(\kappa-4n-1)}\lambda_p^2+p^{4n}+\cdots+p+p^{-1}+\cdots+p^{-4n})&(i=1)\\
|R_{4n+1}^{(i-1)}|(\mu_1-\displaystyle\frac{p^{i-1}-1}{p^i-1}f_{4n+2,1})&(2\le i\le 4n+2)
\end{cases}.
\]
\end{thm}
\begin{proof}
We give just an overview of the proof since it is quite similar to the case of the lifting from the Maass cusp forms. The only difference is the recurrence relation for the Fourier coefficients of the holomorphic cusp form $f$ as follows:
\begin{lem}\label{Hecke-rel-elliptic}
Let $f(\tau)=\sum_{n\ge 1}c(m)q^m\in S_{\kappa-(4n-2)}(SL_2(\Z))$ be a Hecke eigenform. We have
\begin{align*}
c(p^2m)&=(\lambda_p^2-p^{\kappa-(4n-1)})c(m)-
{\begin{cases}
p^{\kappa-(4n-1)}\lambda_pc(m/p)&(p|m)\\
0&(p\nmid m)
\end{cases}},\\
c(p^2m)&=(\lambda_p^2-2p^{\kappa-(4n-1)})c(m)-p^{2(\kappa-(4n-1))}c(m/p^2),
\end{align*}
where we assume {$p^2|m$} for the second formula.  
\end{lem}
This follows from the well known recurrence relation of the Fourier coefficients~(cf.~\cite[Chapitre VII,~Section 5.3, Corollaire 2]{Se}). 

With this lemma and Proposition \ref{Hecke-recurrence} for $\cW_{\lambda}^{\cF}$ on $G_{4n+2}$, we get the explicit formula for $\mu_1$ by the proof similar to that of Proposition \ref{Hecke-eigen-1}. 
The formula for $\mu_i$ with $i\ge 2$ is then an immediate consequence from Proposition \ref{Simple-Heckemodule}.
\end{proof}
\subsection*{Cuspidal representation generated by $F_f$}
We now state the theorem quite similar to Theorem \ref{MainThm-second}.
\begin{thm}\label{cusp-rep-Oda-RS}
Let $\pi_{F_f}$ be the cuspidal representation generated by $F_f$ and suppose that $f$ is a Hecke eigenform.\\
(1)~The representation $\pi_{F_f}$ is irreducible and thus has the decomposition into the restricted tensor product $\otimes'_{v\le\infty}\pi_v$ of irreducible admissible representations $\pi_v$.\\
(2)~For $v=p<\infty$, $\pi_p$ is the spherical constituent of the unramified principal series representation of $G_p$ with the Satake parameter
\[
\diag\left(\left({\frac{\lambda'_p+\sqrt{{\lambda'}_p^2-4}}{2}}\right)^2,p^{4n},\cdots,p,1,1,p^{-1},\cdots,p^{-4n},\left({\frac{\lambda'_p+\sqrt{{\lambda'}_p^2-4}}{2}}\right)^{-2}\right),
\]
where $\lambda'_p:=p^{-\frac{\kappa-(4n-1)}{2}}\lambda_p$.\\
(3)~For every finite prime $p<\infty$, $\pi_p$ is non-tempered while $\pi_{\infty}$ is tempered.
\end{thm}
\begin{proof}
Also for this theorem we sketch the proof since it is very similar to that of Theorem \ref{MainThm-second}. 
Let $G_{\infty}^0$ and $K_{\infty}^0$ be the connected component of the identity for $G_{\infty}$ and $K_{\infty}$ respectively, and let ${\mathfrak g}_{\infty}$ be the Lie algebra of $G_{\infty}$. 
The right translations of $F_f$ by $G_{\infty}^0$ generate the anti-holomorphic discrete series representation $\pi_{\kappa}$ with minimal $K_{\infty}^0$-type given by
\[
K_{\infty}^0\ni k\mapsto J(k,\sqrt{-1}e)^{-\kappa}
\]
as a $({\mathfrak g}_{\infty},K_{\infty}^0)$-module, and it is irreducible. Given a complete set $\{\sigma\}$ of representatives for $G_{\infty}/G_{\infty}^0\simeq K_{\infty}/K_{\infty}^0$, the right translations of $F_f$ by $G_{\infty}$ generate an irreducible $({\mathfrak g}_{\infty},K_{\infty})$-module whose restriction to $G_{\infty}^0$ is a finite direct sum of the conjugations $\pi_{\kappa}^{\sigma}$ by $\sigma$s with $\sigma$ running over the representatives. {According to \cite[Lemma 3.5]{La} the archimedean component of $\pi_F$ is irreducible}.
Since $F_f$ is a Hecke eigenform under the assumption we see the irreducibility of $\pi_F$ by \cite[Theorem 3.1]{N-P-S}. This is nothing but the first assertion. 
In view of Theorem \ref{Hecke-eigenval-Oda-RS} the rest of the assertions are settled by the proof similar to (2) and (3) in Theorem \ref{MainThm-second}. 
For the third assertion we remark that the discrete series representations of semi-simple real Lie groups are a well-known class of tempered representations. 
\end{proof}

As we deduce Corollary \ref{Std-L-fct} from Theorem \ref{MainThm-second} we have the following as an immediate consequence from Theorem \ref{cusp-rep-Oda-RS}.
\begin{cor}\label{Std-L-fct-Oda-RS}
For any prime $p$ the local $p$-factor $L_p(\pi_{F_f},\operatorname{St},s)$ of the standard $L$-function for $\pi_{F_f}$~(or $F_f$) is written as
\begin{align*}
L_p(\pi_{F_f},\operatorname{St},s)=L_p(\sym^2(f),s)\prod_{j=0}^{8n}\zeta_p(s+j-4n),
\end{align*}
where $L_p(\sym^2(f),s)$ is the $p$-factor of the symmetric square $L$-function for $f$.
\end{cor}
\begin{rem}
This result is essentially obtained in \cite[Theorem 8.1]{Su}, which expresses the standard $L$-functions of the Oda-Rallis-Schiffman lifts in the Jacobi form formulation in terms of $L$-functions of Jacobi forms. Sugano has remarked that $L_p(\sym^2(f),s)$ is a local factor of the $L$-function of some Jacobi form. 
\end{rem}
\section*{{\bf Acknowledgement}}
The second named author would like to express his profound gratitude to Prof.~Takashi Sugano and Prof.~Masao Tsuzuki for their comments or discussions related to this study, {especially for the non-archimedean local theory}. 
The second named author was partially supported by Grand-in-Aid for Scientific Research (C) 16K05065, Japan Society for the Promotion of Science. 
This work was supported by the Research Institute for Mathematical Sciences, a Joint Usage/Research Center located in Kyoto University. 
The first named author would like to thank the MPIM at Bonn for organizing the third Japanese-German number theory workshop, when some of the works here were discussed and completed.
The first named author was partially supported by the DFG grant BR-2163/4-2, an NSF postdoctoral fellowship, and the LOEWE research unit USAG.

\newpage
\thispagestyle{empty}
\noindent
Yingkun Li\\
Fachbereich Mathematik\\
Technische Universit{\"a}t Darmstadt\\ 
Schlossgartenstr. 7\\
64289 Darmstadt, Germany\\
{\it E-mail address}:~li@mathematik.tu-darmstadt.de
\\[10pt]
Hiro-aki Narita\\
Department of Mathematics\\
Faculty of Science and Engineering\\
Waseda University\\
3-4-1 Ohkubo, Shinjuku, Tokyo 169-8555, Japan\\
{\it E-mail address}:~hnarita@waseda.jp
\\[10pt]
Ameya Pitale\\
Department of Mathematics\\
University of Oklahoma\\
Norman, Oklahoma, USA\\ 
{\it E-mail address}:~apitale@ou.edu
\end{document}